\documentclass[11pt,letterpaper]{amsart}
\usepackage{latexsym, graphicx, amsmath, amssymb, cite, hyperref, color}

\newtheorem{theorem}{Theorem}[section]
\newtheorem{lemma}[theorem]{Lemma}
\newtheorem{proposition}[theorem]{Proposition}
\newtheorem{corollary}[theorem]{Corollary}

\theoremstyle{definition}

\theoremstyle{remark}
\newtheorem{remark}[theorem]{Remark}

\numberwithin{equation}{section}

\newcommand{\R}{{\mathbb R}}

\newcommand{\Z}{{\mathbb Z}}

\newcommand{\T}{{\mathbb T}}

\renewcommand{\Re}{\operatorname{Re}}
\renewcommand{\Im}{\operatorname{Im}}

\newcommand{\supp}{\operatorname{supp}}

\newcommand{\codim}{\operatorname{codim}}
\newcommand{\vol}{\operatorname{vol}}
\newcommand{\WF}{\operatorname{WF}}
\newcommand{\ord}{\operatorname{ord}}
\newcommand{\rank}{\operatorname{rank}}
\newcommand{\sgn}{\operatorname{sgn}}

\title{A two term Kuznecov sum formula}
\author{Emmett L. Wyman}
\address{Department of Mathematics, University of Rochester, Rochester NY}
\email{emmett.wyman@rochester.edu}

\author{Yakun Xi}
\address{School of Mathematical Sciences, Zhejiang University, Hangzhou 310027, PR China}
\email{yakunxi@zju.edu.cn}
\begin{document}

\maketitle

\begin{abstract}
The Kuznecov sum formula, proved by Zelditch in the Riemannian setting \cite{ZelK}, is an asymptotic sum formula
\[
	N(\lambda) := \sum_{\lambda_j \leq \lambda} \left| \int_H e_j \, dV_H \right|^2 = C_{H,M} \lambda^{\codim H} + O(\lambda^{\codim H - 1})
\]
where $e_j$ constitute a Hilbert basis of Laplace-Beltrami eigenfunctions on a Riemannian manifold $M$ with $\Delta_g e_j = -\lambda_j^2 e_j$, and $H$ is an embedded submanifold. We show for some suitable definition of `$\sim$',
\[
	N(\lambda) \sim C_{H,M} \lambda^{\codim H} + Q(\lambda) \lambda^{\codim H - 1} + o(\lambda^{\codim H - 1})
\]
where $Q$ is a bounded oscillating term and is expressed in terms of the geodesics which depart and arrive $H$ in the normal directions. Our result generalizes a theorem of Safarov on the pointwise Weyl law \cite{Safarov}.

In \cite{CGT, CanGal1}, Canzani, Galkowski, and Toth establish (as a corollary to a stronger result involving defect measures) that if the set of recurrent directions of geodesics normal to $H$ has measure zero, then we obtain improved bounds on the individual terms in the sum---the period integrals. We are able to give a dynamical condition such that $Q$ is uniformly continuous and `$\sim$' can be replaced with `$=$'. This implies improved bounds on period integrals, and this condition is weaker than the recurrent directions having measure zero. Moreover, our result implies improved bounds for period integrals if there is no $L^1$ measure on $SN^*H$ that is invariant under the first return map. This generalizes a theorem of Sogge--Zelditch \cite{SZRev} and Galkowski \cite{galkowski}.
\end{abstract}


\section{Introduction}

Let $(M,g)$ be a compact Riemannian manifold without boundary and $H$ an embedded compact submanifold. Let $e_1,e_2,\ldots$ constitute a Hilbert basis of Laplace-Beltrami eigenfunctions for $L^2(M)$ with
\[
	\Delta_g e_j = -\lambda_j^2 e_j.
\]
A \emph{period integral} is the integral of the restriction of an eigenfunction $e_j$ to $H$, namely
\[
	\int_H e_j \, dV_H,
\]
where here $dV_H$ is the volume element on $H$ induced by the metric $g$. The size of a period integral can be thought of as a measurement of the oscillation of an eigenfunction along $H$.

Period integrals were studied back in the 1920's by Hecke and Petersson as means of understanding Maass forms on compact or finite-area cusped hyperbolic surfaces. Weyl sums of periods are called Kuznecov sum
formulae or local trace formulae which were introduced and studied by Kuznecov \cite{Kuz} and Bruggeman \cite{Brug}. The work of Good \cite{Good} and \cite{Hej} introduced analytical tools into this subject. Period integrals have since received a great deal of attention in the Riemannian setting. The fundamental result in the Riemannian setting is due to Zelditch \cite{ZelK}, who proved, among other things, the asymptotic sum formula
\begin{equation}\label{N def}
	N(\lambda) := \sum_{\lambda_j \leq \lambda} \left| \int_H e_j \, dV_H \right|^2 = C_{H,M} \lambda^{\codim H} + O(\lambda^{\codim H - 1}),
\end{equation}
where 
\[
	C_{H,M} =(2\pi)^{-\codim H} \vol(H) \vol(B^{\codim H}),
\]
where $B^{\codim H}$ is the unit ball in $\R^{\codim H}$.
This implies both general and generic bounds on the period integrals, namely that
\begin{equation}\label{standard period bounds}
	\left| \int_H e_j \, dV_H \right| = O(\lambda_j^\frac{\codim H - 1}{2}),
\end{equation}
and that for every $\epsilon > 0$, there exists a density-$1$ subsequence of eigenfunction basis elements $e_{j_k}$, $k = 1,2,\ldots$ satisfying
\[
	\left| \int_H e_{j_k} \, dV_H \right| = O(\lambda_{j_k}^{-\frac{\dim H}{2} + \epsilon}).
\]
It is conjectured that these generic bounds are in fact general bounds if $M$ is a compact hyperbolic surface and $H$ is a closed geodesic or circle \cite{Rez}, but this seems a long way off.

The remainder of \eqref{N def} and the period integral bounds \eqref{standard period bounds} are sharp in the sense that they cannot be improved when $M$ is the standard sphere (see e.g. \cite{emmett3}). However, there has been an explosion of interest in improving \eqref{standard period bounds} under a variety of geometric and dynamical assumptions on pair $(M,H)$ \cite{CSPer,Gauss,emmett1,emmett2,emmett3,emmett4,CGT,CanGal1,CanGal2,CanGal3,CanGal4}. The first logarithmic improvement over \eqref{standard period bounds} was obtained in \cite{Gauss} where $H$ is a geodesic on a Riemannian surface $M$ . The article \cite{CanGal2} contains state-of-the-art logarithmic improvements over \eqref{standard period bounds} which hold for the broadest assumptions to date. In \cite{CanGal4}, among other things, the authors obtained improved bounds for the discrepancy $N(\lambda)-\rho_{t_0}*N(\lambda)$ between the counting function and its convolution with a suitable Schwartz-class function under suitable dynamical assumptions. While an estimate of this type is essential to obtain improvements to the remainder in \eqref{N def}, it is not altogether sufficient.

While the bounds \eqref{standard period bounds} on individual period integrals of eigenfunctions and quasimodes have seen many improvements, we have only partial results for the remainder term of the corresponding asymptotic formula. This paper shows that an improvement to the remainder term reveals an oscillating term of order $\lambda^{\codim H - 1}$, which we compute in terms of the geometry of geodesics which depart and arrive again at $H$ in normal directions. In the process, we identify weakened sufficient conditions for little-$o$ improvements to the remainder term and to bounds on period integrals of quasimodes.

\subsection{Statement of Results}

We first introduce some convenient notation.

\bigskip

\noindent
\textbf{Notation.} In what follows, we will let $n = \dim M$ and $d = \dim H$. By $f(\lambda) \lesssim g(\lambda)$, we mean there exists a constant $C$ depending only on $M$ and $H$ for which $f(\lambda) \leq C g(\lambda)$ for all $\lambda$. Given two tempered functions $f$ and $g$ where $f$ is monotone-increasing, we write
\[
f(\lambda) \sim g(\lambda) + o(\lambda^{\alpha})
\]
to mean there exists a monotone-decreasing function $\epsilon(\lambda) \to 0$ as $\lambda \to \infty$ such that
\[
g(\lambda - \epsilon(\lambda)) - o(\lambda^\alpha) \leq f(\lambda) \leq g(\lambda + \epsilon(\lambda)) + o(\lambda^\alpha),
\]
where the $o(\lambda^\alpha)$ terms depend on $\epsilon$. If $g$ is uniformly continuous, then the above implies $f(\lambda) = g(\lambda) + o(\lambda^\alpha)$. (See e.g. \cite{Saf}.)

 Our first theorem characterizes the term $Q(\lambda)$ in terms of dynamical properties of the geodesic flow. 
Recall, the principal symbol of the half-Laplacian $\sqrt{-\Delta_g}$ is given by
\[
	p(x,\xi) = |\xi|_{g(x)} = \left(\sum_{i,j} g^{ij}(x) \xi_i \xi_j\right)^{1/2},
\]
where $(x,\xi)$ are canonical local coordinates of $T^*M$. The time-$t$ flow of its Hamilton vector field
\[
	H_p = \sum_{j} \left( \frac{\partial p}{\partial \xi_j} \frac{\partial}{\partial x_j} - \frac{\partial p}{\partial x_j} \frac{\partial}{\partial \xi_j} \right)
\]
is the homogeneous (or analyst's) geodesic flow $G^t : \dot T^*M \to \dot T^*M$. The dynamical features of the flow determine improvements for various spectral asymptotic quantities. The relevant geodesics for \eqref{N def} are those which start and end conormally to $H$. In truth, we will only need to consider the set of such normal geodesics which are nearly stable under perturbation. Precisely, we suppose $(x,\xi) \in \dot N^*H$ and let $t \neq 0$ and suppose $G^t(x,\xi) \in \dot N^*H$. Then, we require that
\begin{equation}\label{structured looping direction}
	dG^t : T_{(x,\xi)} \dot N^*H \to T_{G^t(x,\xi)} \dot N^*H \quad \text{ is a linear isomorphism}.
\end{equation}
We then write
\begin{equation} \label{structured looping set}
	\Sigma_t = \{(x,\xi) \in \dot N^*H : G^t(x,\xi) \in \dot N^*H \text{ and satisfies \eqref{structured looping direction}}\}.
\end{equation}
There is a natural volume density on $\dot N^*H$. To describe it, we select local coordinates $x = (x',x'') \in \R^{d} \times \R^{n-d}$ for which $x'' = 0$ defines $H$, and let $g_H(x')$ denote the local Riemannian metric tensor on $H$. By considering canonical local coordinates $(x,\xi)$ of $T^*M$ and writing $\xi = (\xi',\xi'')$ similarly, we obtain a parametrization of $\dot N^*H$ by $(x',\xi'')$ and we write the natural density as
\[
	\frac{|g_H(x')|}{|g(x)|^\frac12} |dx' \, d\xi''| \qquad x = (x',0).
\]
We let $J_t(x,\xi)$ denote the (positive) determinant of the map $dG^t$ in \eqref{structured looping direction} with respect to these volume elements. That is, if $G^t(x,\xi) = (y,\eta)$ satisfied \eqref{structured looping direction} and $(y,\eta)$ is expressed in similar local coordinates as $(x,\xi)$, then $J_t(x,\xi)$ is locally defined by the identification of the densities
\begin{equation}\label{def J}
	\frac{|g_H(y')|}{|g(y)|^\frac12} |dy' \, d\eta''| \simeq J_t(x,\xi) \frac{|g_H(x')|}{|g(x)|^\frac12} |dx' \, d\xi''|.
\end{equation}
Next, we note $SN^*H = \dot N^*H \cap p^{-1}(1)$, and hence comes equipped with the Leray volume element induced by the condition $p(x,\xi) = 1$. The volume element on $SN^*H$ yields a measure on
\[
	S \Sigma_t := \Sigma_t \cap SN^*H = \{(x,\xi) \in \Sigma_t : p(x,\xi) = 1\}.
\]
It follows that there exists a countable set of non-zero times $\mathcal T\subset\mathbb R\setminus\{0\}$ such that $S\Sigma_t$ has positive measure if and only if $t\in\mathcal T\cup\{0\}.$
\begin{remark}[Convenient local coordinates for the objects above] First, select local coordinates $x' \in \R^d$ of $H$. Then, extend these to coordinates $x = (x',x'')$ of $M$ such that the $x''$ coordinate vectors form an orthonormal frame of vectors, each normal to $H$ at $x'' = 0$. The metric tensor $g$ is then locally written
\[
	g(x) = \begin{bmatrix}
		g_H(x') & 0 \\
		0 & I
	\end{bmatrix} \qquad x = (x',0).
\]
The volume density on $\dot N^*H$ above then reads as
\[
	|g_H(x')|^{1/2} |dx' \, d\xi''|,
\]
and the Leray volume element on $SN^*H = \dot N^*H \cap p^{-1}(1)$ is given by
\[
	|g_H(x')|^{1/2} |dx' \, d\omega|,
\]
where $\omega \in S^{n-d-1}$ denotes the angular part of $\xi'' \in \R^{n-d}$, i.e. $\xi'' = |\xi''| \omega$. The determinant $J_t(x,\xi)$ can also be calculated point-by-point by selecting geodesic normal coordinates $x = (x',x'')$ and $y = (y',y'')$, writing $dG^t : T_{(x,\xi)} \dot N^*H \to T_{(y,\eta)} \dot N^*H$ in matrix form $A : \R^n \to \R^n$. Then, $J_t = |\det A|$.
\end{remark}

Here and throughout, we will define
\begin{equation}\label{Q def}
	Q(\lambda) = \sum_{t\in\mathcal T} e^{-it\lambda} \frac{q(t)}{-it}, \qquad q(t) = (2\pi)^{-n + d} \int_{S \Sigma_t} i^{\sigma_t} \sqrt{J_t}.
\end{equation}
Here, since $t\in\mathcal T,$ the sum is taken over nonzero $t$ for which $q(t) \neq 0$, which is only ever countably many. The integral is taken with respect to the natural measure on $SN^*H$ discussed above, and $\sigma_t$ is a locally constant Maslov factor. It is clear from the definition that $Q$ is a distribution on $\mathbb R$. Nonetheless, we will show that $Q$ is a bounded function with nice properties by exploiting the fact that $q(t)$ is a dynamical object. See Proposition \ref{Q nice}.
	
	Our first main result is the following.

\begin{theorem}\label{Q asymptotics} For $N$ and $Q$ as above, there exists a constant $C$ for which
	\[
	N(\lambda) \sim C_{H,M} \lambda^{n-d} + Q(\lambda) \lambda^{n-d-1} + o(\lambda^{n-d-1}) + C.
	\]
\end{theorem}

The inclusion of the constant $C$ is to address a quirk of the codimension-$1$ case, where the remainder reads $o(1)$. This is because the asymptotics of $N$ will be obtained by integrating (a smoothing of) the measure $N'$ from $0$ to $\lambda$. In this case, even usually negligible rapidly-decaying errors might contribute a constant to the asymptotics. Note, this constant disappears into the remainder if $\codim H \geq 2$.

Theorem \ref{Q asymptotics} and its proof are inspired by the work of Safarov \cite{Safarov} on the study of similar $Q$ terms for the point-wise Weyl law. In particular, it should be compared with Theorem 1.8.14 in \cite{Saf}. Nonetheless, we want to point out that our formulation of $Q$ is in a different form than that in \cite{Saf}, in the sense that we sort the contributions to the second term by time $t$, instead of using the first-return map. As commented by the authors in \cite{Saf}, it seems very difficult to give sufficient conditions under which $Q$ is uniformly continuous. However, following the method in \cite{SZRev}, we are able to give a sufficient condition using our formulation of $Q$. The next theorem provides such a condition. In this case, the notation `$\sim$' becomes equality, and $N$ does not have serious jumps.

\begin{theorem}\label{Q continuous} If $Q$ is uniformly continuous, then
	\[
	N(\lambda) = C_{H,M} \lambda^{n-d} + Q(\lambda) \lambda^{n-d-1} + o(\lambda^{n-d-1}) + C.
	\]
Furthermore, $Q$ is uniformly continuous if 
\begin{equation}\label{averaging condition}
	\lim_{T \to \infty} \frac{1}{T} \sum_{t\in\mathcal T\cap[-T,T]} |q(t)| = 0.
\end{equation}
\end{theorem}
Since $q(t)$ is a dynamical object, the condition \eqref{averaging condition} is imposed on the dynamic of the geodesic flow in some averaged sense. 
As a corollary, we find that \eqref{averaging condition} tells us when we may have an improvement to period integrals for quasimodes with shrinking spectral support. This is discussed in greater detail below in Corollary \ref{quasimode corollary}. This condition can also be related to the set of recurrent directions. See Proposition \ref{recurrent}.

Our last main result improves the asymptotic sum formula \eqref{N def} provided the set of conormal looping directions has measure zero. Compare this, for example, to the Duistermaat-Guillemin theorem \cite{DG, Ivrii}, which improves the remainder term of the Weyl law provided the set of covectors belonging to closed orbits of the geodesic flow is measure zero. This theorem follows directly from \eqref{Q def} and Theorem \ref{Q continuous}.

\begin{theorem}\label{Q zero} $Q$ is constantly zero if and only if $S \Sigma_t$ is a measure-zero subset of $SN^*H$ for each $t \in\mathcal T$, in which case
\[
	N(\lambda) = C_{H,M} \lambda^{n-d} + o(\lambda^{n-d-1}) + C.
\]
\end{theorem}

We will prove Theorems \ref{Q asymptotics} and \ref{Q continuous}, from which Theorem \ref{Q zero} immediately follows. The proof of Theorem \ref{Q asymptotics} is presented in Section \ref{TAUBERIAN}, though the bulk of the computations are deferred until Sections \ref{SINGULARITY AT ZERO}, \ref{SINGULARITIES OFF ZERO} and \ref{Q NICE}. The main term of Theorem \ref{Q asymptotics} arises from the big singularity of the Fourier transform of $N'$ at $0$. The principal part of this singularity is easy to compute, but we also need to show the subprincipal part vanishes. This is done in Section \ref{SINGULARITY AT ZERO} with the help of Lemma \ref{subprincipal lemma}, which is a stationary phase result for real oscillatory integrals with real symbols. The oscillating $Q$ term in the asymptotics is determined by the singularities of $N'$ which occur away form zero. By contrast, these singularities have a principal part which is more difficult to compute, but there is no need to compute their subprincipal terms. This is done in Section \ref{SINGULARITIES OFF ZERO}. Normally, we would use Duistermaat and Guillemin's calculus for Fourier integral operators with cleanly composing canonical relations, but we do not assume cleanness in the theorem. Instead, we rely on a delicate argument based on estimates for oscillatory integrals with phase functions with flat parts (see Lemma \ref{very stationary phase lemma}).


\subsection{Implications of the main results} \label{implications}

Theorem \ref{Q continuous} yields improvements to the standard period integral bounds \eqref{standard period bounds}. In fact, this holds for quasimodes of shrinking spectral support.

\begin{corollary}\label{quasimode corollary} Suppose \eqref{averaging condition} is satisfied, and let $\epsilon(\lambda)$ be any monotone-decreasing function which vanishes in the limit $\lambda \to \infty$. If $\psi_\lambda$ is any sum of eigenfunctions with eigenvalues in the range $[\lambda, \lambda + \epsilon(\lambda)]$ with $\|\psi_\lambda\|_{L^2(M)} = 1$, then we have
\[
	\left| \int_H \psi_\lambda \, dV_H \right| = o(\lambda^\frac{n-d-1}{2}).
\]
\end{corollary}

This corollary follows by applying the asymptotics of Theorem \ref{Q continuous} to the difference $N(\lambda + \epsilon) - N(\lambda)$, and writing $\psi$ as a linear combination of eigenfunctions and applying Cauchy-Schwarz. In other words, the jumps of $N(\lambda)$ are $o(\lambda^{n-d-1})$ if \eqref{averaging condition} holds.

In many ways, this work is inspired by that of Sogge and Zelditch \cite{SZRev} (see also \cite{SZRev2}) which gives necessary conditions for when the standard pointwise estimates
\[
	|e_j(x)| = O(\lambda^\frac{n-1}{2}) \qquad x \in M
\]
are saturated for a real-analytic manifold $M$. By their previous work with Toth \cite{STZ}, saturation can only occur when the set of recurrent directions at $x$ has positive measure in $S_x^*M$. In the case where $M$ is real-analytic, this may only occur when all geodesics departing $x$ arrive again at $x$ after a common period $T > 0$, such as in the case of an umbilic point on a triaxial ellipsoid.

Though they do not highlight the oscillating quantity $Q(\lambda)$ or the dynamical quantities $q(t)$ with definitions, their argument yields Corollary \ref{quasimode corollary} provided $M$ is real-analytic and $H = \{x\}$ is a singleton set consisting of a self-focal point $x$. In this case, $G^{T_0}(S^*_xM) = S^*_xM$ where $T_0 > 0$ is the least common period, they are able to show the condition \eqref{averaging condition} holds if and only if there does not exist any $L^1$ measure on $S_x^*M$ which is invariant under the first return map $G^{T_0} : S^*_x M \to S^*_x M$. This result is later generalized by Galkowski \cite{galkowski} to the smooth case. Our main theorem implies similar results for a general submanifold $H$.  As in \cite{galkowski}, by making use of the first return map, we do not assume the existence of a least common period.

Let $\mathbf T(x,\xi)$ be the smallest positive time $t$ for which $(x,\xi) \in S\Sigma_t$. We say an $L^1$ measure $f(x,\xi) \, dx \, d\xi$ is invariant under the first return map on $SN^*H$ if
\[
	f = \begin{cases} 
		f \circ G^{\mathbf T} J_{\mathbf T} & \text{if } \mathbf T < \infty \\
		0 & \text{if } \mathbf T = \infty.
	\end{cases}
\]

\begin{theorem}\label{invariant function}
	Suppose the only invariant $L^1$ measure of the first return map on $SN^*H$ is the trivial one. Then, \eqref{averaging condition} holds.
\end{theorem}

We are also able to show that, even in the general setting, the condition \eqref{averaging condition} is satisfied when the set of recurrent directions in $SN^*H$ is measure zero. It would be interesting to know if \eqref{averaging condition} is strictly weaker. To state this precisely, we recall (e.g. from \cite{STZ}) that the set of recurrent directions $\mathcal R$ are those $(x,\xi) \in SN^*H$ such that, for each open neighborhood $U$ of $(x,\xi)$ in $SN^*H$, there exists $t \neq 0$ with $G^t(x,\xi) \in U$.

\begin{proposition} \label{recurrent} If the set $\mathcal R$ of recurrent directions has measure zero in $SN^*H$, then \eqref{averaging condition} is satisfied.
\end{proposition}

\subsection{Examples}

We give two examples to illustrate the theorems above. First, let $M = \T^2 = \R^2/2\pi \Z^2$ be the standard two-dimensional flat torus. There is a natural choice of eigenbasis, and that is the one consisting of $L^2$-normalized Fourier exponentials,
\[
	e_m(x) = (2\pi)^{-1} e^{i\langle x, m \rangle} \qquad m \in \Z^2.
\]
We let $H$ here denote the unit distance circle about $0$. Note, the period integrals are given by
\[
	\int_H e_m \, dV_H = \frac{1}{2\pi} \int_0^{2\pi} e^{i|m|\sin \theta} \, d\theta = J_0(|m|),
\]
where $J_0$ is the Bessel function with index $0$. Using the asymptotics of the Bessel function, one can show that $N$ has asymptotics
\[
	N(\lambda) = \sum_{|m| \leq \lambda} |J_0(|m|)|^2 = 2 \lambda - \cos(2\lambda) + o(1) + C
\]
where $C$ comes from an accumulation of low-order errors. The graph of $N(\lambda)$ against the main term $2\lambda$ is depicted in Figure \ref{bessel1}. Curiously, the graph suggests $C = 0$, and it would be interesting to know if this is the case in general for $n = 2$ and $d = 1$.
\begin{figure}
\includegraphics[width=0.8\textwidth]{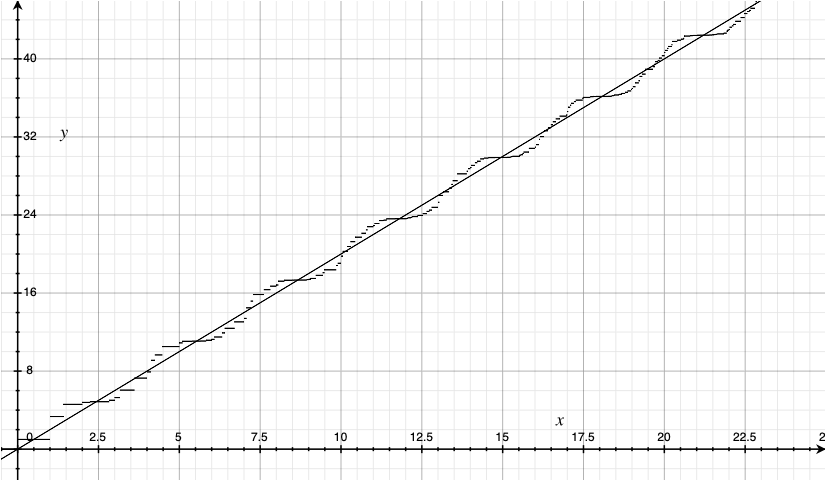}
\caption{A graph of $N(\lambda)$ vs. $2\lambda$ in the example $M = \mathbb T^2
	$ and $H = S^1$.}
\label{bessel1}
\end{figure}

The oscillating term $\cos(2\lambda)$ is reflected in the structure of the normal geodesics to $H$. The diameters of the circle yield a positive measure set of elements in $S\Sigma_t$ with $|t| = 2$. Specifically, $S\Sigma_2$ consists of all inward-pointing unit normals, and $S\Sigma_{-2}$ consists of all outward-pointing unit normals. Furthermore, $S\Sigma_t$ has measure zero 
for $|t| \neq 2$ or $0$, and thus $\mathcal T=\{-2,2\}$. Theorems \ref{Q asymptotics} and \ref{Q continuous} show that we should expect an oscillating second term with angular frequency $2$.

The second example is where $M$ is a triaxial ellipsoid (again, $\dim M = 2$). Let $x_0$ be an umbilical point. Recall, every geodesic departing $x_0$ passes through its antipode $-x_0$ and loops back to $x_0$ at a common time, say $T_0$. Now, take $H$ to be the circle of points which are equidistant from $x_0$ and $-x_0$ --- namely the circle of radius $T_0/4$ centered at $x_0$. By Gauss' lemma, any geodesic departing $H$ in the normal direction necessarily passes through $x_0$. Furthermore, expressing $H$ as the circle of radius $T_0/4$ in geodesic normal coordinates about $x_0$ allows us to deduce dynamics on $SN^*H$ from the dynamics on $S_{x_0}^*M$. In particular, $S\Sigma_t = SN^*H$ if $t$ is an integer multiple of $T_0/2$ and empty otherwise. Furthermore, the recurrent directions in $SN^*H$ have measure zero. Hence, Theorems \ref{Q asymptotics} and \ref{Q continuous} and Proposition \ref{recurrent} yield nice asymptotics for the corresponding counting function $N(\lambda)$.

\subsection*{Acknowledgements} Both authors are grateful to Steve Zelditch for suggesting the problem and for providing his invaluable input on an early draft of the paper. Steve Zelditch also communicated to the authors that he initiated a project in studying the two-term asymptotics for several other cases, including ladder sums for  Fourier coefficients of restricted eigenfunctions as in \cite{WXZ1,WXZ2}, and the asymptotics for $L^2$ restrictions of eigenfunctions. The authors want to thank an anonymous referee for many helpful historical comments.


\section{Proof of Theorem \ref{Q asymptotics}} \label{TAUBERIAN}

We will require the following three propositions, which are proved later in Sections \ref{SINGULARITY AT ZERO}, \ref{SINGULARITIES OFF ZERO} and \ref{Q NICE}, respectively. Their proofs make up the bulk of the work. 

The first proposition computes the big singularity of $\widehat N'$ at the origin.
\begin{proposition} \label{big singularity prop} Let $\rho$ be a real-valued Schwartz-class function on $\R$ with small Fourier support, with $\widehat \rho$ real-valued and $\widehat \rho(0) = 1$. Then,
	\[
	N' * \rho(\lambda) = (2\pi)^{-n+d} \vol(H) \vol(S^{n-d-1}) \lambda^{n - d - 1} + O(\lambda^{n - d - 3}).
	\]
	It follows
	\[
	N * \rho(\lambda) = C_{H,M} \lambda^{n-d} + O(\lambda^{n-d-2} \log \lambda) + C,
	\]
	where the $\log \lambda$ is present to cover the case $n - d = 2$ and the constant term $C$ is here to cover the case $n - d = 1$.
\end{proposition}
The second proposition concerns the singularities away from the origin.
\begin{proposition}\label{off zero singularities} Let $\chi$ be a Schwartz-class function on $\R$ with compact Fourier support in $\R \setminus 0$. Then,
	\[
	N' * \chi(\lambda) = \lambda^{n-d-1} \sum_{t\in\mathcal T} \widehat \chi(-t) e^{-it\lambda} q(t) + o(\lambda^{n-d-1}),
	\]
	where $q$ is as in \eqref{Q def}.
\end{proposition}
Our last proposition guarantees that $Q$ is a well-defined function and has nice properties. Interestingly, the proof of this proposition relies heavily on that $q(t)$ is a dynamical object.
\begin{proposition}\label{Q nice}
	$Q$ is a bounded function such that $(n-d)C_{H,M} \lambda + Q(\lambda)$ is monotone increasing.
\end{proposition}

Our main tool will be the refined Tauberian theorem in Appendix B of \cite{Saf}, stated below for convenience.

\begin{theorem}[Theorem B.4.1 of \cite{Saf}] \label{tauberian}
	Let $N_1, N_2$ be tempered, monotone-increasing functions on $\R$ which vanish on the negative real numbers, and let $\rho$ be nonnegative, Schwartz-class functions on $\R$ with compact Fourier support and $\int \rho = 1$. If
	\begin{equation}\label{tauberian 1}
		N_1 * \rho(\lambda) = N_2 * \rho(\lambda) + o(\lambda^\alpha)
	\end{equation}
	and, for every Schwartz-class function $\chi$ with compact Fourier support in $\R \setminus 0$,
	\begin{equation}\label{tauberian 2}
		N_1' * \chi(\lambda) = N_2' * \chi(\lambda) + o(\lambda^\alpha),
	\end{equation}
	then,
	\[
	N_1(\lambda) \sim N_2(\lambda) + o(\lambda^\alpha).
	\]
\end{theorem}

First, note that Proposition \ref{Q nice} imply that $Q$ is bounded and satisfies $Q(\lambda) + c \lambda$ is monotone-increasing for all $\lambda > 0$ for some $c > 0$.
To obtain the asymptotics of Theorem \ref{Q asymptotics}, we compare two monotone-increasing functions
\[
N_1(\lambda) = N(\lambda) + c\mathbf 1_{[0,\infty)}(\lambda) \lambda^{n-d}
\]
and
\[
N_2(\lambda) = \mathbf 1_{[0,\infty)}(\lambda)( C_{H,M} \lambda^{n-d} + \lambda^{n-d-1}(Q(\lambda) + c\lambda) + C),
\]
where here $\mathbf 1_{[0,\infty)}(\lambda)$ denotes the Heaviside function, and where the constant $C$ is as in Proposition \ref{big singularity prop}.

We now list some facts that we will need in order to apply Theorem \ref{tauberian}. They are quite routine to verify, so they are largely left to the reader. In what follows, we use
\[
p_k(\lambda) = \lambda^k
\]
to denote the degree-$k$ power function on $\R$. If $\rho$ in Theorem \ref{tauberian} is also even, then one quickly verifies
\begin{equation}\label{tauberian ingredient 1}
	(\mathbf 1_{[0,\infty)} p_{n-d}) * \rho(\lambda) = \lambda^{n-d} + O(\lambda^{n-d-2}) \qquad \lambda \geq 0.
\end{equation}
(To see why, remove the Heaviside function at the expense of a rapidly decaying error and write the Taylor expansion of $p_k$ to two terms about $\lambda$.) We also have
\begin{align*}
	(p_{n-d-1} Q) * \rho(\lambda) &= \int_\R \tau^{n-d-1} Q(\tau) \rho(\lambda - \tau) \, d\tau \\
	&= \frac{1}{2\pi} \int_\R  \widehat Q(t) \left(\frac{1}{-i} \frac{d}{dt}\right)^{n-d-1}{\big( e^{it\lambda}\widehat \rho(t) \big)}\, dt \\
	&= 0,
\end{align*}
where the second line follows from the Fourier inversion formula and the last line can be forced by taking the Fourier support of $\rho$ to be small and disjoint from that of $Q$. Hence,
\begin{equation}\label{tauberian ingredient 2}
	(\mathbf 1_{[0,\infty)} p_{n-d-1} Q) * \rho(\lambda) = (p_{n-d-1} Q) * \rho(\lambda) + O(\lambda^{-\infty}) = O(\lambda^{-\infty}) \qquad \lambda \geq 0.
\end{equation}
Similarly for $\chi$, we have
\begin{equation}\label{tauberian ingredient 3}
	(\mathbf 1_{[0,\infty)} p_{n-d-1}) * \chi(\lambda) = p_{n-d-1} * \chi(\lambda) + O(\lambda^{-\infty}) = O(\lambda^{-\infty}) \qquad \lambda \geq 0.
\end{equation}
Finally, we have
\begin{equation}\label{tauberian ingredient 4}
	(\mathbf 1_{[0,\infty)} p_{n-d-1} Q') * \chi(\lambda) = \lambda^{n-d-1} Q' * \rho(\lambda) + O(\lambda^{n-d-2}) \qquad \lambda \geq 0.
\end{equation}
To see this, we remove the Heaviside function at the expense of a rapidly decaying error and write
\begin{multline*}
	(p_{n-d-1} Q') * \chi(\lambda) \\
	= \lambda^{n-d-1} Q' * \chi(\lambda) - \int_\R Q'(\tau) (\lambda^{n-d-1} - \tau^{n-d-1}) \chi(\lambda - \tau) \, d\tau.
\end{multline*}
Then, use integration by parts and the boundedness of $Q$ to show the integral on the right is $O(\lambda^{n-d-2})$.

With our facts assembled, \eqref{tauberian 1} follows from Proposition \ref{big singularity prop}, \eqref{tauberian ingredient 1}, and \eqref{tauberian ingredient 2}, while \eqref{tauberian 2} follows from Proposition \ref{off zero singularities}, \eqref{tauberian ingredient 3}, and \eqref{tauberian ingredient 4}. The conclusion of Theorem \ref{tauberian} then yields
\[
N(\lambda) \sim C_{H,M} \lambda^{n-d} + \lambda^{n-d-1}Q(\lambda) + o(\lambda^{n-d-1}) + C
\]
after subtracting out the $c \lambda^{n-d}$ term.


\section{Proof of Theorem \ref{Q continuous}}

That `$\sim$' can be replaced with `$=$' in the asymptotics of $N$ if $Q$ is uniformly continuous is straightforward. So, we must prove:

\begin{proposition}\label{Q continuous prop}
	Let $Q$ be given as in Theorem \ref{Q asymptotics}. Then, $Q$ is uniformly continuous if 
	\[
	\lim_{T \to \infty} \frac{1}{T} \sum_{t\in\mathcal T\cap[-T,T]} |q(t)| = 0.
	\]
\end{proposition}

This proposition is based on the arguments in \cite{SZRev}, who prove it for $q$ supported on the integers. Similar arguments works in general.

\begin{proof} Suppose the limit in the proposition holds. Let $c > 0$ be a positive constant such that $F(\lambda) = c\lambda + Q(\lambda)$ is monotone increasing (see Proposition \ref{Q nice}). Then, it suffices to show that $F$ is uniformly continuous. Let $\rho$ be an even, nonnegative Schwartz-class function $\R$ for which $\widehat \rho$ is nonnegative and supported in $[-1,1]$. In particular, we will assume $\rho \geq 1$ on $[-1,1]$. Then,
\[
F(\lambda + \epsilon) - F(\lambda - \epsilon) \leq \int_{-\epsilon}^\epsilon F'(\lambda + \tau) \, d\tau \leq \int 
F'(\lambda + \tau) \rho(\tau/\epsilon) \, d\tau = \epsilon F' * \rho_\epsilon(\lambda),
\]
where in the last line we set $\rho_\epsilon(\tau) = \epsilon^{-1} \rho(\tau/\epsilon)$. Now, using the definition of $Q$ before Theorem \ref{Q asymptotics}, we write
\[
\epsilon F' * \rho_\epsilon(\lambda) = c\epsilon \int \rho(t) \, dt +
\epsilon \sum_{t\in\mathcal T} e^{-it\lambda} q(t) \widehat \rho(\epsilon t).
\]
Hence, we have for some $C > 0$,
\[
F(\lambda + \epsilon) - F(\lambda - \epsilon) \leq C \epsilon \left( 1 + \sum_{0 < |t| \leq \epsilon^{-1}} |q(t)| \right),
\]
which vanishes as $\epsilon \to 0$ uniformly in $\lambda$ by hypothesis.
\end{proof}


\section{Proof of Proposition \ref{big singularity prop}: Analysis of the big singularity at zero} \label{SINGULARITY AT ZERO}

\subsection{Stationary Phase for Real Oscillatory Integrals}

We state and prove a simple stationary phase lemma for later application to certain real valued oscillatory integrals with vanishing sub-principle symbols.

Let $a(x,y,\lambda)$ be a symbol of order $m$ with compact support in $(x,y)\in\R^n\times\R^e$ such that  
\[
a(x,y,\lambda) = a_0(x,y) \lambda^m + a_1(x,y) \lambda^{m-1} + a_2(x,y,\lambda),
\]
where $a_2(x,y,\lambda) \in S^{m-2}$, and $a_0(x,y)$ is real-valued.
Let $\phi$ be smooth on $\supp a$ such that $\phi(0,0) = 0$, $d_x \phi(0,0) = 0$, and $\det d_x^2 \phi(0,0) \neq 0$. By the implicit function theorem, there exists a smooth function $y \mapsto x(y)$ so that $d_x \phi(x(y),y) = 0$.

\begin{lemma}\label{subprincipal lemma} Consider the oscillatory integral
\[
	I(y,\lambda) = \int_{\R^n} e^{i\lambda \phi(x,y)} a(x,y,\lambda) \, dx
\]
with $a$ and $\phi$ as above. Then,
\begin{multline}
	I(y,\lambda) = (2\pi)^{\frac{n}{2}} \lambda^{m- \frac{n}{2}} e^{\frac{\pi i \sigma}{4}} e^{i\lambda \phi(x(y),y)}|\det d_x^2 \phi(x(y),y)|^{-1/2}\\
	\cdot(a_0(x(y),y) + \lambda^{-1} \Re a_1(x(y),y)+\lambda^{-1}r(y))  + O(\lambda^{m-\frac{n}{2}-2}),
\end{multline}
where $\Re r(y)=0.$
\end{lemma}

Lemma \ref{subprincipal lemma} follows from a standard stationary phase result in \cite{HI} which we state below.
\begin{lemma}[Theorem 7.7.6 of \cite{HI}]\label{H lemma}
	Let $a(x,y)$ be smooth and compactly supported on $\R^n \times \R^e$, and let $\phi$ be smooth on $\supp a$ such that $\phi(0,0) = 0$, $d_x \phi(0,0) = 0$, and $\det d_x^2 \phi(0,0) \neq 0$. By the implicit function theorem, there exists a smooth function $y \mapsto x(y)$ so that $d_x \phi(x(y),y) = 0$. Consider the oscillatory integral indexed by
	\[
		I(y,\lambda) = \int_{\R^n} e^{i\lambda \phi(x,y)} a(x,y) \, dx.
	\]
	Then,
	\begin{multline*}
		I(y,\lambda) = (2\pi)^{\frac{n}{2}} \lambda^{-\frac{n}{2}} e^{\frac{\pi i \sigma}{4}} e^{i\lambda \phi(x(y),y)} |\det d_x^2 \phi(x(y),y)|^{-1/2}\\
		 \cdot \Big( a(x(y),y) + \lambda^{-1} L_y a(x(y),y) + O(\lambda^{-2})  \Big)
	\end{multline*}
	where
	\begin{multline*}
		L_y a = i^{-1} \Big( \frac12 \langle d_x^2 \phi(x(y),y) D_x, D_x \rangle a(x(y),y) \\
		+ \frac{1}{2^{2} 2!} \langle d_x^2 \phi(x(y),y) D_x, D_x \rangle^2 (g_y a)(x(y),y) \\
		+ \frac{1}{2^{3} 2! 3!} \langle d_x^2 \phi(x(y),y) D_x, D_x \rangle^3 (g_y^2 a)(x(y),y) \Big),
	\end{multline*}
	where
	\[
		g_y(x) = \phi(x,y) - \phi(x(y),y) - \frac12 \langle d_x^2 \phi(x(y),y) (x - y), x - y \rangle.
	\]
\end{lemma}

Note that if $a$ is real-valued, $L_y a$ is imaginary. We will exploit this to avoid having to compute any of the terms of $L_y a$. We are ready to proceed with the proof of Lemma \ref{subprincipal lemma}.

\begin{proof}[Proof of Lemma \ref{subprincipal lemma}] 
Applying Lemma \ref{H lemma} to $I(y,\lambda)$, we see that
\begin{multline}
	I(y,\lambda) = (2\pi)^{\frac{n}{2}} \lambda^{m- \frac{n}{2}} e^{\frac{\pi i \sigma}{4}} e^{i\lambda \phi(x(y),y)}|\det d_x^2 \phi(x(y),y)|^{-1/2}\\
	\cdot(a_0(x(y),y) + \lambda^{-1}( a_1(x(y),y) +L_y(a_0(x(y),y))))  + O(\lambda^{m-\frac{n}{2}-2}),
\end{multline}
Since $L_ya_0$ is imaginary, the lemma follows if we take $r(y)=i\Im a_1(x(y),y) +L_y(a_0(x(y),y))$.
\end{proof}

\subsection{Proof of Proposition \ref{big singularity prop}}

We first note that
\[
	N'*\rho(-\lambda) = \sum_j \rho(\lambda_j + \lambda) \left|\int_H e_j \, dV_H \right|^2 = O(\lambda^{-\infty})
\]
since $\rho$ is Schwartz-class and the periods $\left|\int_H e_j \, dV_H \right|^2$ are tempered by standard bounds \eqref{N def}. Hence,
\begin{align*}
	N' * \rho(\lambda) &= N' * \rho(\lambda) + N'*\rho(-\lambda) + O(\lambda^{-\infty}) \\
	&= \sum_j ( \rho(\lambda_j - \lambda) + \rho(\lambda_j + \lambda) ) \left|\int_H e_j \, dV_H \right|^2 + O(\lambda^{-\infty}).
\end{align*}
By Fourier inversion,
\begin{align*}
	\rho(\lambda_j - \lambda) + \rho(\lambda_j + \lambda) &= \frac{1}{2\pi} \int_\R ( e^{-it(\lambda - \lambda_j)} + e^{-it(\lambda + \lambda_j)}) \widehat \rho(t) \, dt \\
	&= \frac{1}{\pi} \int_\R e^{-it\lambda} \cos(t\lambda_j) \widehat \rho(t) \, dt.
\end{align*}
Hence, modulo rapidly-decaying terms,
\begin{align}
	\nonumber N' * \rho(\lambda) &= \sum_j \frac1\pi \int_\R e^{-it\lambda} \cos(t\lambda_j) \widehat \rho(t) \left|\int_H e_j \, dV_H \right|^2 \, dt\\
	\nonumber &= \frac1\pi \int_\R \int_H \int_H e^{-it\lambda}\widehat \rho(t) \sum_j \cos(t\lambda_j) e_j(x) \overline{e_j(y)} \, dV_H(x) \, dV_H(y) \, dt \\
	\label{pre-parametrix} &= \frac1\pi \int_\R \int_H \int_H e^{-it\lambda}\widehat \rho(t) \cos(tP)(x,y) \, dV_H(x) \, dV_H(y) \, dt.
\end{align}
Here, $\cos(tP)$ denotes the the cosine wave operator, which gives the solution $u(t) = \cos(tP)f$ to the wave equation
\[
	(\partial_t^2 - \Delta_g)u = 0
\]
with initial conditions
\begin{align*}
	u(0) &= f \\
	\partial_t u(0) &= 0.
\end{align*}
We may write the distribution kernel of $\cos(tP)$ in local oscillatory form using the Hadamard parametrix. Let $x$ and $y$ be expressed in local coordinates, and let $\xi$ be an element in $\R^n$ such that $(y,\xi)$ form canonical local coordinates of $T^*M$. By \cite[equation (5.2.16)]{Hang}, for $|t|$ less than the injectivity radius of $M$, we write
\begin{multline} \label{hadamard parametrix}
	\cos(tP)(x,y) = (2\pi)^{-n} \int_{\R^n} e^{i\varphi(x,y,\xi)} \cos(tp(y,\xi)) \alpha_0(x,y) \, d\xi \\
	+ \sum_\pm \int_{\R^n} e^{i\varphi(x,y,\xi) \pm i tp(y,\xi)} c_\pm(t,x,y,\xi) \, d\xi + R(t,x,y)
\end{multline}
where $\alpha_0$ is smooth and $\alpha_0(x,x) = 1$ for all $x$, where $\varphi$ is smooth, positive-homogeneous of degree $1$ in the $\xi$ variable, and
\begin{equation} \label{varphi bound}
	\varphi(x,y,\xi) = \langle x - y, \xi \rangle + O(|x - y|^2 |\xi|).
\end{equation}
The terms $c_\pm$ in the second term belong to symbol class $S^{-2}$, meaning
\begin{equation}\label{symbol bounds c}
	|\partial_{t,x,y}^\alpha \partial_\xi^\beta c_\pm(t,x,y,\xi)| \leq C_{\alpha,\beta} ( 1 + |\xi|)^{-2 - |\beta|}
\end{equation}
for all multiindices $\alpha$ and $\beta$. The remainder $R(t,x,y)$ can be taken to be in class $C^N$ for a finite but arbitrarily large number $N$. Since $t$ will be confined to the support of $\widehat \rho$, say so that $|t| < \delta$, we may assume each of $\alpha_0$, $c_\pm$, and $R$ are supported for $d_g(x,y) \leq \delta$. After all, for $|t| < \delta$, $\cos(tP)$ is supported on $d_g(x,y) \leq |t| < \delta$ by H\"uygen's principle, and so we may multiply both sides by a smooth cutoff in $(x,y)$ which takes the value $1$ for $d_g(x,y) \leq \delta$, and $0$ for $d_g(x,y) \geq 2\delta$.

First, we deal with the contribution of the remainder term $R$ to \eqref{pre-parametrix}. Note, as a function of $t$,
\[
	\int_H \int_H \widehat \rho(t) R(t,x,y) \, dV_H(x) \, dV_H(y)
\]
may have as many continuous derivatives as we like. Hence, integrating by parts in $t$ yields a contribution of 
\[
	\frac1\pi \int_\R \int_H \int_H e^{-it\lambda}\widehat \rho(t) R(t,x,y) \, dV_H(x) \, dV_H(y) \, dt = O(\lambda^{-N})
\]
for a fixed, arbitrarily large $N$. Hence, the contribution of the remainder term is negligible, and we are left to estimate the contribution of the first two terms of \eqref{hadamard parametrix}. First, we consider the contribution of the `$-$' term of the sum in the second term, namely
\[
	(2\pi)^{-n-1} \int_{\R^n} \int_\R \int_H \int_H e^{i(\varphi(x,y,\xi) - tp(y,\xi) - t\lambda)} \widehat \rho(t) c_-(t,x,y,\xi) \, dV_H(x) \, dV_H(y) \, dt \, d\xi.
\]
After making a change of variables $\xi \mapsto \lambda \xi$, we write this contribution as
\[
	(2\pi)^{-n-1} \lambda^n \int_{\R^n} \int_\R \int_H \int_H e^{i\lambda (\varphi(x,y,\xi) - t(p(y,\xi) + 1))} \widehat \rho(t) c_-(t,x,y,\lambda \xi) \, dV_H(x) \, dV_H(y) \, dt \, d\xi.
\]
Integration by parts $N$ times in $t$, along with the symbol bounds \eqref{symbol bounds c}, yields (1) an amplitude which is $L^1$ (uniformly in $\lambda \geq 1$) if $N$ is at least $n-1$, and (2) an improvement of the power of $\lambda$ in front of the integral to $\lambda^{n - N}$. By taking $N$ plenty large, we see the contribution of this term is negligible. We also write the first term of \eqref{hadamard parametrix} as
\begin{multline*}
	(2\pi)^{-n} \int_{\R^n} e^{i\varphi(x,y,\xi)} \cos(tp(y,\xi)) \alpha_0(x,y) \, d\xi \\
	= (2\pi)^{-n} \frac{1}{2} \sum_\pm \int_{\R^n} e^{i\varphi(x,y,\xi) \pm itp(y,\xi)} \alpha_0(x,y) \, d\xi,
\end{multline*}
and similarly argue that the contribution of the `$-$' term is negligible. We have thus reduced the proof of Proposition \ref{big singularity prop} to the following claim:
\begin{multline*}
	(2\pi)^{-n-1} \int_{\R^n} \int_\R \int_H \int_H e^{i(\varphi(x,y,\xi) + tp(y,\xi) - t\lambda)} \\
	\times \widehat \rho(t) (\alpha_0(x,y) + 2c_+(t,x,y,\xi)) \, dV_H(x) \, dV_H(y) \, dt \, d\xi \\
	= (2\pi)^{-n+d} \vol(H) \vol(S^{n-d-1}) \lambda^{n - d - 1} + O(\lambda^{n - d - 3}).
\end{multline*}
By making a change of variables $\xi \mapsto \lambda \xi$, we write the left side as
\begin{multline*}
	(2\pi)^{-n-1} \lambda^n \int_{\R^n} \int_\R \int_H \int_H e^{i\lambda (\varphi(x,y,\xi) + t(p(y,\xi) - 1))} \\
	\times \widehat \rho(t) (\alpha_0(x,y) + 2c_+(t,x,y,\lambda \xi)) \, dV_H(x) \, dV_H(y) \, dt \, d\xi.
\end{multline*}
Let $\beta$ be a smooth function on $\R$ taking values in the interval $[0,1]$, where $\beta \equiv 1$ on a neighborhood of $1$ with support in $[1/2, 2]$. We then cut the integral into $\beta(p(y,\xi))$ and $1 - \beta(p(y,\xi))$ pieces. By a similar integration by parts argument as before, the latter piece contributes $O(\lambda^{-\infty})$ to the whole. Hence, we have reduced our claim to the following oscillatory integral asymptotics:
\begin{multline}\label{big singularity claim}
	(2\pi)^{-n-1} \int_{\R^n} \int_\R \int_H \int_H e^{i\lambda \psi(t,x,y,\xi)} a(t,x,y,\xi; \lambda) \, dV_H(x) \, dV_H(y) \, dt \, d\xi \\
	= (2\pi)^{-n+d} \vol(H) \vol(S^{n-d-1}) \lambda^{n - d - 1} + O(\lambda^{n - d - 3}),
\end{multline}
with phase function 
\[
	\psi(t,x,y,\xi) = \varphi(x,y,\xi) + t(p(y,\xi) - 1)
\]
and amplitude
\[
	a(t,x,y,\xi;\lambda) = \lambda^n \widehat \rho(t) \beta(p(y,\xi)) (\alpha_0(x,y) + 2c_+(t,x,y,\lambda \xi)).
\]

We are about ready to apply our stationary phase tool, Lemma \ref{subprincipal lemma}. To do so, we will set up some specific local coordinates. We fix $y \in H$ and fix geodesic normal coordinates about $y$ with respect to which the plane $\{(x',0) \in \R^n : x' \in \R^d \}$ is tangent to $H$ at $y$. We also express the covector $\xi$ locally as $\xi = (\xi', r\omega)$ where $\xi' \in \R^d$, $r > 0$, and $\omega \in S^{n-d-1}$. Since we have fixed geodesic normal coordinates about $y$, and since $\xi$ is a covector at $y$, we have $p(y,\xi) = |\xi|$. Finally, let $\Phi$ be smooth map from an open neighborhood of $\R^d$ into $\R^n$ so that $\Phi(x')$ parametrizes $H$ in our local coordinates, and assume further that $\Phi(0) = 0$ and $d\Phi(0) = [I \ 0]$.

With this setup along with \eqref{varphi bound}, the phase function $\psi$ reads
\[
	\langle x', \xi' \rangle + t(|\xi| - 1) + O(|x'|^2 |\xi|).
\]
Fixing $\omega$ and taking derivatives in $t, r, x',$ and $\xi'$ variables yields
\[
	\nabla_{t, r, x',\xi'} \psi = \begin{bmatrix}
		|\xi| - 1 \\
		tr/|\xi| + O(|x'|^2) \\
		\xi' + O(|x'||\xi|) \\
		x' + t\xi'/|\xi| + O(|x'|^2)
	\end{bmatrix}.
\]
Note this derivative vanishes at $(t,r,x',\xi') = (0,1,0,0)$. Furthermore, at such a critical point, we have Hessian matrix
\[
	\nabla^2_{t, r, x', \xi'} \psi = \begin{bmatrix}
		0 & 1 & 0 & 0 \\
		1 & 0 & 0 & 0 \\
		0 & 0 & * & I \\
		0 & 0 & I & 0
	\end{bmatrix},
\]
which is nonsingular. Hence, the condition of Lemma \ref{subprincipal lemma} is satisfied if the support of $\widehat \rho$ is sufficiently small and the amplitude is supported on a correspondingly small neighborhood of the diagonal $x = y$. On the other hand, the amplitude $a(t,x,y,\xi;\lambda)$ has principal part $(2\pi)^{-n-1} \lambda^n \widehat \rho(t) \beta(p(y,\xi)) \alpha_0(x,y)$, vanishing subprincipal part, and remainder $2 \lambda^n \widehat \rho(t) \beta(p(y,\xi)) c_+(t,x,y,\lambda\xi)$, which one quickly verifies is a symbol of order $n-2$ in $\lambda$. Now we apply our stationary phase Lemma \ref{subprincipal lemma} to $I(\lambda)$, we have 
\begin{align*}
	I(\lambda) &= (2\pi)^{-n+d} e^{\frac{\pi i \sigma}{4}}\int_H \int_{SN_y^*H}1\,d\omega\,dV_H(y) \lambda^{n - d - 1} + O(\lambda^{n-d-3})\\
	&=(2\pi)^{-n+d} {\rm vol}(H)\, {\rm vol}(S^{n-d-1})\, \lambda^{n - d - 1} + O(\lambda^{n-d-3}).
\end{align*}
Here we have used the fact that $I(\lambda)$ is real and nonnegative modulo lower order terms, so the second term vanishes and the Maslov factor $e^{\frac{\pi i \sigma}{4}}$ is $1$.


\section{Proof of Proposition \ref{off zero singularities}: Analysis of the highest order singularities off zero} \label{SINGULARITIES OFF ZERO}

\subsection{The reduction}

We now study the highest-order singularities of $\widehat{N'}$ which lie away from the origin. For this, we prepare our problem for the global theory of Fourier integral operators.

Let $\delta_H$ denote the half-density distribution on $M$ given by the action
\[
	f |dV_M|^{1/2} \mapsto \int_H f \, dV_H,
\]
where here $dV_M(x) = |g|^{1/2} \, dx$ denotes the Riemannian volume element on $M$ and $dV_H(x') = |g_H(x')|^\frac12$ denotes the induced Riemannian volume element on $H$. Furthermore, let $U(t,x,y)$ be the half-density distribution kernel of the half-wave operator $e^{-itP}$, which we formally express as
\[
	U(t,x,y) = \sum_j e^{-it\lambda_j} e_j(x)\overline{e_j(y)} |dV_M(x) \, dV_M(y) \, dt|^{1/2}.
\]
By interpreting $U$ as the operator taking half-density distributions on $M$ to half-density distributions on $\R \times M$, we write
\[
	\widehat{N'}(t) |dt|^{1/2} = (U(\delta_H)(t), \delta_H).
\]
Hence, we will need to compute the symbolic data of $U(\delta_H)$. This begins with the symbolic data of $\delta_H$.

\begin{lemma}\label{delta_H symbol}
	$\delta_H$ is a Lagrangian distribution in class $I^{\frac n 4 - \frac d 2}(M, \dot N^*H)$ with principal symbol having half-density part
	\[
		(2\pi)^{-\frac n 4 + \frac d 2} \frac{|g_H(x')|^{1/2}}{|g(x)|^{1/4}} |dx' \, d\xi''|^{1/2}
	\]
	where $(x,\xi) = (x',x'',\xi',\xi'')$ are canonical local coordinates and $x'' = 0$ defines $H$, and where $g_H$ denotes the restriction of the Riemannian metric $g$ on $M$ to $H$.
\end{lemma}

\begin{proof}
	Choose local coordinates $x = (x',x'') \in \R^n$, with $x' \in \R^d$ such that $x'' = 0$ defines $H$. Then if $f |dV_M|^{1/2}$ is a test half-density with support in the coordinate patch, we write it as $f(x) |dV_M(x)|^{1/2} = f(x) |g(x)|^{1/4} \, |dx|^{1/2}$, where $g$ is the local Riemannian metric. We then have
	\begin{align*}
		(\delta_H, f |dV_M|^{1/2}) &= \int_{\R^d} f(x',0) |g_H(x')|^\frac12 \, dx' \\
		&= (2\pi)^{-n + d} \int_{\R^{n-d}} \int_{\R^n} e^{i\langle x'', \xi'' \rangle} f(x) |g_H(x')|^\frac12 \, dx \, d\xi''.
	\end{align*}
	Hence, we express $\delta_H$ locally as an oscillatory integral
	\[
		\delta_H(x) = (2\pi)^{-\frac{3n}{4} + \frac d 2} \left( \int_{\R^{n-d}} e^{i\langle x'', \xi'' \rangle} (2\pi)^{-\frac n 4 + \frac d 2} \frac{|g_H(x')|^\frac12}{|g(x)|^{\frac14}} \, d\xi'' \right) \, |dx|^{1/2}.
	\]
	The phase function $\phi(x,\xi'') = \langle x'', \xi'' \rangle$ is nondegenerate and parametrizes $\dot N^* H = \{(x',0,0,\xi'') : x' \in \R^d, \ \xi'' \in  \R^{n-d} \setminus 0\}$. The half-density part of the invariant principal symbol is given by
	\[
		(2\pi)^{-\frac n 4 + \frac d 2} \frac{|g_H(x')|^\frac12}{|g(x)|^{\frac14}} \sqrt{d_\phi} = (2\pi)^{-\frac n 4 + \frac d 2} \frac{|g_H(x')|^\frac12}{|g(x)|^{\frac14}} |dx' \, d\xi''|^\frac12,
	\]
	where
	\begin{align*}
		d_\phi &= \left|\frac{\partial(x',\xi'',\phi'_{\xi''})}{\partial(x',x'',\xi'')} \right|^{-1} |dx' \, d\xi''| = \left|
		\det
		\begin{bmatrix} 
			I & 0 & 0 \\
			0 & 0 & I \\
			0 & I & 0
		\end{bmatrix}
		\right|^{-1} |dx' \, d\xi''| = |dx' \, d\xi''|
	\end{align*}
	is the Leray density on $\phi'_{\xi''}(x',x'',\xi'') = 0$. This completes the proof.
\end{proof}

We also recall from \cite{HormanderIV} the symbolic data of $U$. First, $U \in I^{-1/4}(\R \times M \times M; \mathcal C')$ where $\mathcal C$ is the canonical relation
\[
	\mathcal C = \{(t, -p(x,\xi), G^t(x,\xi) ; x,\xi) : (x,\xi) \in \dot T^*M \}
\]
with principal symbol with half-density part
\[
	(2\pi)^{1/4} |dx|^{1/2}|d\xi|^{1/2} |dt|^{1/2}.
\]
The following lemma describes the symbolic data of the composition of $U$ and $\delta_H$. One can obtain this from the standard calculus for transversally composing canonical relations as in \cite{HormanderIV, HormanderPaper, DuistermaatFIOs}. This can also be obtained as a special case of the broader calculation in Lemma 3.1 of \cite{emmett5}.

\begin{lemma}\label{U symbolic data} $U(\delta_H) \in I^{\frac{n}{4} - \frac{d}{2} - \frac{1}{4}}(\R \times M, \Lambda)$ where
\[
	\Lambda = \{(t, -p(x,\xi), G^t(x,\xi)) : (x,\xi) \in \dot N^*H\}.
\]
Given the parametrization of $\dot N^*H$ by $(x',\xi'')$ as before, we have that the half-density part of the principal symbol of $U(\delta_H)$ is
\[
	 (2\pi)^{-\frac n 4 + \frac d 2 + \frac 14} |dt|^{1/2} \frac{|g_H(x')|^{1/2}}{|g(x)|^{1/4}} |dx' \, d\xi''|^{1/2}.
\]
\end{lemma}

Next, let $S$ be the operator with kernel $U(\delta_H)$ taking half-density distributions on $M$ to half-density distributions on $\R$. Then, 
\[
	\widehat{N'}(t) |dt|^{1/2} = S \circ \delta_H.
\]
By the standard calculus, we have
\[
	\WF (S \circ \delta_H) = \Lambda' \circ \dot N^*H = \{(t,p(x,\xi)) : G^t(x,\xi)) \in \dot N^*H, \ (x,\xi) \in \dot N^*H \}.
\]
That is, if $(t,\tau) \in \WF(S \circ \delta_H)$, then there exists a geodesic of length $t$ which departs and arrives at $H$ perpendicularly at both ends. We first deal with the case where there is at least a little transversality in the composition. In what follows, we consider a slice
\[
	\Lambda_{t_0} = \{G^{t_0}(x,\xi) : (x,\xi) \in \dot N^*H \}
\]
of $\Lambda$ for some fixed $t_0$. If we set $G^t(x,\xi) = (y,\eta) \in \dot N^*H$, we note that the condition \eqref{structured looping direction} is equivalent to the condition
\begin{equation}\label{maximal intersection}
	T_{(y,\eta)} \Lambda_t = T_{(y,\eta)} \dot N^*H.
\end{equation}
To this end, we consider a pseudodifferential partition of unity $1 = A + B$ on $\R \times M$ modulo a smooth operator where $A$ is microlocalized to a small conic neighborhood of points $(t,p(y,\eta),y,\eta)$ at which \eqref{maximal intersection} is satisfied, and where $B$ is microlocally supported away from such a neighborhood. We then write
\[
	S = S_A + S_B
\]
where $S_A$ and $S_B$ are the operators with kernels $A \circ U(\delta_H)$ and $B \circ U(\delta_H)$, respectively. We then decompose $\widehat{N'}|dt|^{1/2} = S_A(\delta_H) + S_B(\delta_H)$. The contribution of $S_A$ to $\chi * N'$ is then
\[
	(2\pi)^{-1} \int_\R e^{it\lambda} \widehat \chi(t) S_A(\delta_H)(t) \, dt,
\]
and similarly for $S_B$.

\begin{lemma}\label{B lem} The contribution of $S_B(\delta_H)$ to $\chi * N'$ in the asymptotics of Proposition \ref{off zero singularities} is $O(\lambda^{n-d-3/2})$.
\end{lemma}

\begin{lemma}\label{A lem} The contribution of $S_A(\delta_H)$ to $\chi * N'$ in the asymptotics of Proposition \ref{off zero singularities} is 
\[
	\lambda^{n-d-1} \sum_{t\in\mathcal T} \widehat \chi(-t) e^{-it\lambda} q(t) + o(\lambda^{n-d-1}).
\]
\end{lemma}

Lemma \ref{B lem} is made possible by the `partial transversality' of the composition $\WF'(S_B) \circ \dot N^*H$. On the other hand, Lemma \ref{A lem} owes its proof in part to the extra structure induced by the condition \eqref{maximal intersection}.

\subsection{Proof of Lemma \ref{B lem}}

Fix canonical local coordinates $(y,\eta)$ of $\dot T^*M$ so that $y'' = 0$ defines $H$. Then, we write the kernel of $S_B$ in local oscillatory form as
\[
	\left( \int_{\R^N} e^{i \phi(t,y,\theta)} a(t,y,\theta) \, d\theta \right) |dt \, dy|^{1/2}
\]
where here $\phi$ is a nondegenerate homogeneous phase function parametrizing $\Lambda$ via
\[
	\{(t,\phi'_t, y, \phi'_y) : \phi'_\theta = 0\} \subset \Lambda,
\]
and where $a$ is a polyhomogeneous symbol of order $\frac{n-d-N}{2}$ by Lemma \ref{U symbolic data} and 
\[
	\ord S_B = \ord a + \frac{N}{2} - \frac{n+1}{4}.
\]
By a partition of unity, we write the contribution of $S_B$ to $N' * \chi(\lambda)$ as a finite sum of oscillatory integrals of the form
\[
	\int_\R \int_{\R^d} \int_{\R^N} e^{i (\phi(t,y,\theta) + t\lambda)} a(t,y,\theta) \, d\theta \, dy' \, dt
\]
where we are to understand $y$ as $(y',0)$ in the integral above. By a change of variables $\theta \mapsto \lambda \theta$, we have
\[
	= \lambda^N \int_\R \int_{\R^d} \int_{\R^N} e^{i \lambda(\phi(t,y,\theta) + t)} a(t,y,\lambda \theta) \, d\theta \, dy' \, dt.
\]
We cut this integral by $\beta(p(y,\phi'_y))$ and $1 - \beta(p(y,\phi'_y))$ where $\beta$ is a smooth function with $\beta \equiv 1$ about a neighborhood of $1$ and $\beta \equiv 0$ outside $(1/2,2)$. The latter contributes a $O(\lambda^{-\infty})$ term by an integration by parts argument in the $t$ variable, and by using that $\phi'_t = - p(y,\phi'_y) = - 1$ on the critical point
\[
	0 = d_{t,y',\theta}(\phi(t,y,\theta) + t) = \begin{bmatrix}
		- p(y,\phi'_y) + 1 \\
		\phi_{y'}' \\
		\phi_\theta'
	\end{bmatrix}.
\]
We are now left to estimate 
\[
	\lambda^{\frac{n - d + N}{2}} \int_\R \int_{\R^d} \int_{\R^N} e^{i \lambda (\phi(t,y,\theta) + t)} b(t,y,\theta, \lambda) \, d\theta \, dy' \, dt
\]
where
\[
	b(t,y,\theta,\lambda) = \lambda^{-\frac{n-d-N}{2}} a(t,y,\lambda \theta) \beta(p(y,\phi'_y))
\]
is a polyhomogeneous symbol of order $0$ with $\lambda$ as the conic variable. To prove Lemma \ref{B lem}, it suffices to show $\rank \phi'' \geq N - (n - d) + 3$ on the critical set of $\phi(t,y,\theta) - t$. To this end, suppose $(t_0,y_0',\theta_0)$ is in the critical set. By a rotation, assume $\theta_0 = (|\theta_0|, 0,\ldots,0)$, and assume by homogeneity that $p(y_0,\phi'_y(t_0,y_0,\theta_0)) = 1$. Taking $\theta = (\theta_1,\theta')$, we have differential
\[
	\phi_{t,\theta_1, y', \theta'}' = \begin{bmatrix}
		- p(y,\phi_y') + 1 \\
		\frac{\theta_1}{|\theta|} \phi(t,y,\frac{\theta}{|\theta|}) \\
		\phi'_{y'} \\
		\phi'_{\theta'}
	\end{bmatrix}
\]
and, at $(t_0,y_0',\theta_0)$, we have Hessian
\begin{equation}\label{hessian 1}
	\phi_{t,\theta_1, y', \theta'}'' = \begin{bmatrix}
		0 & -1 & * & * \\
		-1 & 0 & 0 & 0 \\
		* & 0 & \phi_{y'y'}'' & \phi_{y'\theta'}'' \\
		* & 0 & \phi_{\theta'y'}'' & \phi_{\theta'\theta'}''
	\end{bmatrix}
\end{equation}
Hence, it suffices to show the lower right block has rank at least $N - (n-d) + 1$. 

We note that, for fixed $t = t_0$, $\phi(t_0,y,\theta)$ is a nondegenerate phase function parametrizing $\Lambda_{t_0}$, i.e. $(y,\phi'_y(t_0,y,\theta)) \in \Lambda_{t_0}$. Then,
\[
	T_{(y_0,\eta_0)} \Lambda_{t_0} = \{ (\partial y, \phi_{yy}'' \partial y + \phi_{y\theta}'' \partial \theta) : \phi_{\theta y}'' \partial y + \phi_{\theta \theta}'' \partial \theta = 0 \}.
\]
Reading off the $\theta_1, y', \theta'$ submatrix from \eqref{hessian 1}, we find
\[
	\rank \phi_{y',\theta}'' = \rank \phi_{y',\theta'}'' \leq N - (n-d).
\]
On the other hand, the submatrix
\begin{equation}\label{differential matrix}
	d_{y',\theta} \phi'_\theta = \begin{bmatrix}
		\phi''_{\theta y'} & \phi''_{\theta\theta}
	\end{bmatrix}
\end{equation}
of $\phi_{y',\theta}''$ is obtained by eliminating $n-d$ columns from the rank-$N$ matrix $d_{y',\theta} \phi'_\theta$, and hence has rank at least $N - (n-d)$, and hence both \eqref{differential matrix} and $\phi''_{y',\theta}$ have rank exactly $N - (n - d)$. We conclude that the span of the missing columns $\phi''_{y'' \theta}$ and the span of \eqref{differential matrix} have trivial intersection. Then, the condition
\begin{equation}\label{kernel condition}
	\phi''_{\theta y} \partial y + \phi''_{\theta\theta} \partial \theta = 0
\end{equation}
reads
\[
	\phi''_{\theta y'} \partial y' + \phi''_{\theta\theta} \partial \theta = -\phi''_{\theta y''} \partial y'',
\]
and we conclude that both left and right sides must vanish. Moreover, $\rank \phi''_{\theta y''} = n-d \leq N$, and hence is injective, and so $\partial y'' = 0$. Moreover, we also have that the top rows of $\phi''_{y',\theta}$ are linear combinations of the bottom rows. Namely, there exists a matrix $A$ for which
\[
	\phi''_{y'y'} = A \phi''_{\theta y'}  \qquad \text{ and } \qquad  \phi''_{y'\theta} = A \phi''_{\theta \theta}.
\]
In particular, if \eqref{kernel condition} is satisfied,
\[
	\phi''_{y'y} \partial y + \phi''_{y'\theta} \partial \theta = \phi''_{y'y'} \partial y' + \phi''_{y'\theta} \partial \theta = A (\phi''_{\theta y'} \partial y' + \phi''_{\theta \theta} \partial \theta) = 0.
\]
We conclude that $T_{(y_0,\eta_0)} \Lambda_{t_0} \subset T_{(y_0,\eta_0)} \dot N^* H$, and hence equality holds. This contradiction completes the proof.

\subsection{Proof of Lemma \ref{A lem}} 

We start by writing down the key oscillatory integral estimates.

\begin{lemma}\label{very stationary phase lemma}
	Let $a$ be a function on $\R^n \times \R_+$, $a_0 \in L^1(\R^n)$, and $m$ a positive integer so that
	\[
		\|a(x,\lambda) - \lambda^m a_0(x) \|_{L^1(dx)} = o(\lambda^m).
	\]
	Let $C_t$ denote the set of all points $x$ in $\phi^{-t}(t)$ such that the intersection of $\phi^{-1}(t)$ with every open neighborhood of $x$ has positive Lebesgue measure. Then,
	\[
		\int_{\R^n} e^{i\lambda \phi(x)} a(x,\lambda) \, dx = \lambda^m \sum_t e^{it\lambda} \int_{C_t} a_0(x) \, dx + o(\lambda^m).
	\]
\end{lemma}

\begin{remark} \label{very stationary remark}
Regarding the points in $C_t$, we note that the phase function $\phi$ in Lemma \ref{very stationary phase lemma} satisfy that all of its derivatives must vanish on $C_t$. Otherwise, $\phi^{-1}(t)$ will be supported on a lower-dimensional variety near a point in $C_t$. We also note that $C_t$ is closed and differs from $\phi^{-1}(t)$ on a set of measure zero. It follows that
\[
	C_t \subset \{x : \phi(x) = t, \ \phi'(x) = 0, \text{ and } \phi''(x) = 0 \} \subset \phi^{-1}(t),
\]
and the difference between any two such sets has zero measure in $\R^n$.

Regarding the stationary values $t$, the only nonzero terms in the sum on the right are those for which $\{\phi = t\}$ has positive Lebesgue measure. We then see that the sum is countable and absolutely convergent.
\end{remark}

\begin{proof}[Proof of Lemma \ref{very stationary phase lemma}] First, by the $L^1$ bounds in the hypothesis,
\begin{multline*}
	\left| \int_{\R^n} e^{i\lambda \phi(x)} a(x,\lambda) \, dx - \lambda^m \int_{\R^n} e^{i\lambda \phi(x)} a_0(x) \, dx \right| \\
	\leq \|a(x,\lambda) - \lambda^m a_0(x) \|_{L^1(dx)} = o(\lambda^m),
\end{multline*}
hence it suffices to show
\[
	\int_{\R^n} e^{i\lambda \phi(x)} a_0(x) \, dx = \sum_t e^{it\lambda} \int_{\phi^{-1}(t)} a_0(x) \, dx + o(1).
\]
We note that $\phi^{-1}(t) \setminus C_t$ has measure zero in $\R^n$, so we exclude it from the domain of integration at no loss, and the lemma will follow. To this end, write
\[
	f(t) = \int_{\{\phi \leq t\}} a_0(x) \, dx.
\]
Now $f$ has bounded variation (by $\|a_0\|_{L^1}$), and hence the distributional derivative $df$ is a signed Radon measure on $\R$, which can be written as a sum of singular and absolutely continuous parts as
\[
	df = \sum_{t : |\phi^{-1}(t)| > 0} \left( \int_{\phi^{-1}(t)} a_0(x) \, dx \right) \, \delta_t + g \, dt,
\]
where $g \in L^1(\R)$. By the layer cake formula, we write
\begin{align*}
	\int_{\R^n} e^{i\lambda \phi(x)} a_0(x) \, dx &= \int_\R e^{it\lambda} df(t) = \sum_t e^{it\lambda} \lambda^m \int_{\phi^{-1}(t)} a_0(x) \, dx + \int_{\R} e^{it\lambda} g(t) \, dt
\end{align*}
where the last integral vanishes in the limit $\lambda \to \infty$ by the Riemann-Lebesgue lemma.
\end{proof}

\begin{lemma} \label{variable very stationary phase lemma}
	Let $a(x,y,\lambda)$ be a polyhomogeneous symbol of order $m$ on $\R^n \times \R^d \times [0,\infty)$ with compact $(x,y)$-support and principal part $a_0(x,y) \lambda^m$. Let $\phi(x,y)$ be a smooth phase function on the support of $a$ such that
	\[
		\det \phi_{xx}'' \neq 0 \qquad \text{ whenever } \qquad \phi'_x = 0.
	\]
	Furthermore, suppose the critical set $\phi'_x = 0$ is parametrized by $(x(y), y)$ where $x(y)$ is a smooth function of $y$ in the $y$-support of $a$. Let $C_t$ be the set of points $y_0$ such that $\phi(x(y),y) = t$ on a positive-measure subset of every open neighborhood of $y_0$.
	Then,
	\begin{multline*}
		\int_{\R^n} \int_{\R^d} e^{i\lambda \phi(x,y)} a(x,y,\lambda) \, dy \, dx \\
		= (2\pi)^{\frac d 2} \lambda^{m - \frac d2} e^{\frac{\pi i \sigma}{4}} \sum_t e^{it\lambda} \int_{C_t} a_0(x(y),y) |\det \phi_{xx}''(x(y),y)|^{-\frac12} \, dy + o(\lambda^{m - \frac d2})
	\end{multline*}
	where $\sigma = \sgn \phi''_{xx}$.
\end{lemma}

The proof of the lemma is a straightforward application of the method of stationary phase in the $x$-variables followed by an application of Lemma \ref{very stationary phase lemma} in the $y$-variables.

\begin{remark} \label{variable very stationary remark}
	Following Remark \ref{very stationary remark}, we have that
	\[
		C_t \subset \{ y : \text{ at } (x(y),y), \ \phi = t, \ \phi' = 0, \text{ and } \rank \phi'' = d\} \subset \{ y : \phi(x(y),y) = t \},
	\]
	and that the difference between any two of the sets has zero measure in $\R^n$.
\end{remark}

With Lemma \ref{variable very stationary phase lemma} in hand, we turn now to the proof of Lemma \ref{A lem}. Let $(y_0,\eta_0)$ be a point in the intersection of $\Lambda_{t_0} \cap \dot N^*H$ satisfying \eqref{maximal intersection}. Furthermore, we express $(y,\eta)$ in canonical local coordinates about $T^*M$ such that $y_0$ corresponds to $0$. We also write as before $y = (y',y'') \in \R^d \times \R^{n-d}$ so that $y'' = 0$ defines $H$. Then by \eqref{maximal intersection}, the projection $(y,\eta) \mapsto (y',\eta'')$ as a map from $\Lambda_t \to \R^d \times \R^{n-d}$ has surjective differential at $(t,y,\eta) = (t_0,y_0,\eta_0)$, and hence also for $(t,y,\eta)$ in some neighborhood of $(t_0,y_0,\eta_0)$. We then have that the projection $(t,y,\eta) \mapsto (t,y',\eta'')$ has surjective differential as a map $\Lambda \to \R \times \R^d \times \R^{n-d}$. Hence, $\Lambda$ is a Lagrangian section over the variables $(t,y',\eta'')$, and so there exists a positive-homogeneous function $\psi(t,y',\eta'')$ for which
\begin{equation}\label{A lem lagrangian}
	\{(t,\tau ,y,\eta) : y'' = \psi'_{\eta''}, \ \eta' = -\psi'_{y'}, \ \tau = -\psi_t' \}
\end{equation}
parametrizes $\Lambda$. In particular, this means that
\begin{equation} \label{t derivative of psi is p}
	\psi_t'(t,y',\eta'') = p(y,\eta) \qquad \text{ for each } (t,y',\eta'').
\end{equation}

\begin{lemma} \label{minimal rank lemma} For $(y,\eta) \in \Lambda_t \cap \dot N^*H$, the condition \eqref{maximal intersection} is satisfied if and only if $\rank \psi''(t,y',\eta'') = 2$.
\end{lemma}

\begin{proof}
We require some convenient coordinates about $(y,\eta)$. First, we let the coordinate vectors $\partial_{y_{d+1}}, \ldots, \partial_{y_n}$ form an orthonormal frame of vectors normal to $H$. Then if $\eta' = 0$, as it is when $(y,\eta) \in \dot N^*H$, we have
\[
	p(y,\eta) = |\eta''|.
\]
Furthermore, we assume by an appropriate rotation of the coordinates that
\[
	\eta'' = (0,\ldots,0,|\eta''|).
\]
We observe from \eqref{A lem lagrangian} that
\[
	\Lambda_t = \{(y,\eta) : y'' = \psi_{\eta''}' \text{ and } \eta' = - \psi_{y'}' \},
\]
where all functions on the right are evaluated at $t$. Now, if $(y,\eta) \in \Lambda_t \cap \dot N^*H$, then \eqref{A lem lagrangian} implies
\[
	\psi'_{\eta''} = 0 \qquad \text{ and } \qquad \psi_{y'}' = 0. 
\]
In what follows, we will write $\tilde \eta'' = (\eta_{d+1} ,\ldots, \eta_{n-1})$ to be the vector with all but the last coordinate of $\eta''$. By the homogeneity of $\psi$, we have
\begin{align*}
	\psi''_{t \eta_n} &= \partial_{\eta_n} p(y, \eta) = 1, \\
	\psi''_{y' \eta_n} &= \psi_{y'}' = 0, \\
	\psi''_{\tilde \eta'' \eta_n} &= \psi'_{\eta''} = 0, \text{ and } \\
	\psi''_{\eta_n \eta_n} &= 0.
\end{align*}
We proceed with this information in hand.

By differentiating the conditions in $\Lambda_t$, we find that $T_{(y,\eta)} \Lambda_t$ is the set of vectors $(\partial y, \partial \eta)$ satisfying
\[
	\begin{bmatrix}
			\partial y'' \\
			\partial \eta'
	\end{bmatrix}
	=
	\begin{bmatrix}
			\psi''_{\eta'' y'} & \psi''_{\eta'' \eta''} \\
			-\psi''_{y' y'} & -\psi''_{y' \eta''}
	\end{bmatrix}
	\begin{bmatrix}
			\partial y' \\
			\partial \eta''
	\end{bmatrix}.
\]
Note, \eqref{maximal intersection} is satisfied if and only if $\partial y'' = 0$ and $\partial \eta' = 0$, and hence if and only if matrix on the right vanishes. But in our chosen coordinates, the matrix is written
\[
	\begin{bmatrix}
			\psi''_{\tilde \eta'' y'} & \psi''_{\tilde \eta'' \tilde \eta''} & \psi''_{\tilde \eta'' \eta_n} \\
			\psi''_{\eta_n y'} & \psi''_{\eta_n \tilde \eta''} & \psi''_{\eta_n \eta_n} \\
			-\psi''_{y' y'} & -\psi''_{y' \tilde \eta''} & \psi''_{y' \eta_n}
	\end{bmatrix}
	=
	\begin{bmatrix}
			\psi''_{\tilde \eta'' y'} & \psi''_{\tilde \eta'' \tilde \eta''} & 0 \\
			0 & 0 & 0 \\
			-\psi''_{y' y'} & -\psi''_{y' \tilde \eta''} & 0
	\end{bmatrix},
\]
and hence
\[
	\text{\eqref{maximal intersection} is satisfied if and only if } \begin{bmatrix}
		\psi''_{y'y'} & \psi''_{y' \tilde \eta''} \\
		\psi''_{\tilde \eta'' y'} & \psi''_{\tilde \eta'' \tilde \eta''}
	\end{bmatrix}
	=0.
\]
At the same time, we have Hessian matrix
	\[
		\psi'' = \begin{bmatrix}
			0 & p'_{y'} & p'_{\eta''} \\
			p'_{y'} & \psi''_{y' y'} & \psi''_{y' \eta''}  \\
			p'_{\eta''} & \psi''_{\eta'' y'} & \psi''_{\eta'' \eta''}
		\end{bmatrix},
	\]
	which, when written in our chosen coordinates, reads as
	\[
		\begin{bmatrix}
			0 & 0 & 0 & 1 \\
			0 & \psi''_{y' y'} & \psi''_{y' \tilde \eta''} & 0  \\
			0 & \psi''_{\tilde \eta'' y'} & \psi''_{\tilde \eta'' \tilde \eta''} & 0 \\
			1 & 0 & 0 & 0
		\end{bmatrix},
	\]
	and so has rank $2$ if and only if the center block vanishes.
\end{proof}

We are now ready to express the kernel of $S_A$ in local oscillatory form in a conic neighborhood of $\zeta_0 = (t_0, p(y_0,\eta_0), y_0, \eta_0)$. We have, modulo a smooth error,
\[
	S_A(t,y) = (2\pi)^{-\frac14 - \frac{3n}{4} + \frac{d}{2}} \left( \int_{\R^{n-d}} e^{i(\langle y'', \eta'' \rangle - \psi(t,y',\eta''))} a(t,y,\eta'') \, d\eta'' \right) |dt|^{1/2} |dy|^{1/2}.
\]
One quickly checks that $\phi(t,y,\eta'') = \langle y'', \eta'' \rangle - \psi(t,y',\eta'')$ is a nondegenerate homogeneous phase function parametrizing $\Lambda$.

We now resolve the principal part of the symbol $a$.

\begin{lemma} \label{local S_A}
	The symbol $a$ in the local expression of $S_A(t,y)$ above is in class $S^0$ and has principal homogeneous part
	\[
		a(t,y,\eta'') = i^{\sigma_{-t}} (2\pi)^{-\frac n 4 + \frac d 2 + \frac 14} \sqrt{J_{-t}(y,\eta)} \frac{|g_H(y')|^\frac12}{|g(y)|^\frac14}
	\]
	at $(t,y,\eta)$ satisfying \eqref{maximal intersection}, and where $\sigma_{-t}$ is the corresponding Maslov factor.
\end{lemma}

\begin{proof}
	Let $\phi(t,y,\eta'') = \langle y'', \eta'' \rangle - \psi(t,y',\eta'')$ be the phase function with critical set defined by $y'' - \psi_{\eta''}' = 0$, parametrized by $t,y',\eta''$, which has Leray density
	\begin{align*}
		d_{\phi} &= \left|\frac{d(t,y',\eta'',\phi_{\eta''}')}{d(t,y',y'',\eta'')}\right|^{-1} |dt \, dy' \, d\eta''| \\
		&= \left| \det \begin{bmatrix}
			1 & 0 & 0 & 0 \\
			0 & I & 0 & 0 \\
			0 & 0 & 0 & I \\
			* & * & I & *
		\end{bmatrix} \right|^{-1} |dt \, dy' \, d\eta''| \\
		&= |dt \, dy' \, d\eta''|.
	\end{align*}
	and where the principal symbol of the kernel of $S_A$ is given by
	\[
		a \sqrt{d_\phi} = a |dt \, dy' \, d\eta''|^{1/2} \qquad \text{ for } y'' = \psi_{\eta''}'.
	\]
	We observe that
	\[
		\ord S_A = \ord a - \frac{n+1}{4} + \frac{n-d}{2} = \ord a + \frac{n}{4} - \frac{d}{2} - \frac14.
	\]
	Since $\ord S_A = \ord U(\delta_H) = \frac{n}{4} - \frac{d}{2} - \frac14$ by Lemma \ref{U symbolic data}, we conclude $\ord a = 0$.
	
	Next, suppose $(y,\eta) = G^t(x,\xi)$ with $(x,\xi) \in SN^*H$ and satisfies \eqref{maximal intersection}, i.e. satisfying \eqref{structured looping direction}. Recalling Lemma \ref{U symbolic data}, this symbol coincides with (modulo a Maslov factor),
	\[
		 (2\pi)^{-\frac n 4 + \frac d 2 + \frac 14} \frac{|g_H(x')|^{1/2}}{|g(x)|^{1/4}} |dt \, dx' \, d\xi''|^{1/2}.
	\]
	Reversing \eqref{def J}, we find that
	\[
		\frac{|g_H(x')|^{1/2}}{|g(x)|^{1/4}} |dx' \, d\xi''| = \sqrt{J_{-t}(y,\eta)} \frac{|g_H(y')|^{1/2}}{|g(y)|^{1/4}} |dy' \, d\eta''|.
	\]
	Hence,
	\[
		a(t,y,\eta'') |dt \, dy' \, d\eta''|^\frac12 = (2\pi)^{-\frac n 4 + \frac d 2 + \frac 14} \sqrt{J_{-t}(y,\eta)} \frac{|g_H(y')|^\frac12}{|g(y)|^\frac14} |dt \, dy' \, d\eta''|^{1/2},
	\]
	and we have
	\[
		a(t,y,\eta'') = (2\pi)^{-\frac n 4 + \frac d 2 + \frac 14} \sqrt{J_{-t}(y,\eta)} \frac{|g_H(y')|^\frac12}{|g(y)|^\frac14}
	\]
	as desired.
\end{proof}

At this point, it is convenient to select local coordinates $y$ so that $y'' = 0$ defines $H$, and furthermore the coordinate vectors $\partial_{y_{d+1}}, \ldots, \partial_{y_n}$ form an orthonormal frame perpendicular to $H$. Then,
\begin{equation} \label{local coordinate p}
	p(y,\eta) = |\eta''| \qquad \text{ for } (y,\eta) \in \dot N^*H.
\end{equation}
and furthermore, since
\[
	g(y) = \begin{bmatrix}
		g_H(y') & 0 \\
		0 & I
	\end{bmatrix} \qquad \text{ for } y'' = 0,
\]
we have
\[
	\frac{|g_H(y')|^\frac12}{|g(y)|^\frac14} = |g_H(y')|^\frac14.
\]
Now, we write locally
\[
	S_A(\delta_H)(t) = (2\pi)^{-\frac14 - \frac{3n}{4} + \frac{d}{2}} \left(\int_{\R^d} \int_{\R^{n-d}} e^{-i \psi(t,y',\eta'')} a(t,y,\eta'') |g_H(y')|^\frac14 \, d\eta'' \, dy' \right) |dt|^\frac12,
\]
where $y = (y',0)$. (There is a minor lie in this equation, since we have omitted writing a sum over a partition of unity.) Now, the contribution of $S_A$ to $\chi * N(\lambda)$ is
\begin{align*}
	&(2\pi)^{-1} \int_\R e^{it\lambda} \widehat \chi(t) S_A(\delta_H)(t) \, dt \\
	&= (2\pi)^{-\frac54 - \frac{3n}{4} + \frac{d}{2}} \int_\R \int_{\R^d} \int_{\R^{n-d}} e^{i(\lambda t - \psi(t,y',\eta''))} a(t,y,\eta'') \widehat \chi(t) |g_H(y')|^\frac14 \, d\eta'' \, dy' \, dt\\
	&= (2\pi)^{-\frac54 - \frac{3n}{4} + \frac{d}{2}} \lambda^{n-d} \int_\R \int_{\R^d} \int_{\R^{n-d}} e^{i\lambda (t - \psi(t,y',\eta''))} a(t,y,\lambda \eta'') \widehat \chi(t) |g_H(y')|^\frac14 \, d\eta'' \, dy' \, dt.
\end{align*}
We change to polar coordinates $\eta'' = r \omega$ with $r > 0$ and $\omega \in S^{n-d-1}$ and obtain
\begin{multline} \label{tr stationary phase prep}
	= (2\pi)^{-\frac54 - \frac{3n}{4} + \frac{d}{2}} \lambda^{n-d} \int_\R \int_{\R^d} \int_0^\infty \int_{S^{n-d-1}} e^{i\lambda \phi(t,r,y',\omega)} a(t,y,\lambda r \omega) r^{n-d-1} \\
	\times \widehat \chi(t) |g_H(y')|^\frac14 \, d\omega \, dr \, dy' \, dt
\end{multline}
where we have set
\[
	\phi(t,r,y',\omega) = t - r\psi(t,y',\omega).
\]
We now aim to apply Lemma \ref{variable very stationary phase lemma}, where $(t,r)$ constitute the nondegenerate variables and $(y',\omega)$ the rest. Fixing $y'$ and $\omega$ and considering $\phi$ as a function of $t$ and $r$, we find that at a critical point
\[
	\phi'_{(t,r)} = \begin{bmatrix}
		1 - r\psi'_t \\
		\psi
	\end{bmatrix} = 0, 
\]
we have a Hessian
\[
	\phi''_{(t,r)} = \begin{bmatrix}
		* & \psi'_t \\
		\psi'_t & 0
	\end{bmatrix}
\]
which is nondegenerate (recall \eqref{t derivative of psi is p}). Hence, after perhaps refining our partition of unity, we ensure all such $(t,r)$ critical points lie on a single implicitly-defined curve $(t(y',\omega), r(y',\omega))$. 

We note that, at a critical point of $\phi$, we have $\phi = t$. So, for fixed $t_0$, we have by Remark \ref{variable very stationary remark}
\[
	C_{t} \subset \{(y',\omega) : \phi = t_0, \ \phi' = 0, \text{ and } \rank \phi'' = 2 \},
\]
and the difference between these two sets has measure zero in $\R^d \times S^{n-d-1}$. By Lemma \ref{minimal rank lemma}, $(y',\omega) \in C_t$ if and only if the corresponding data $(t,y,\eta)$ satisfies \eqref{maximal intersection}, and hence also $G^{-t}(y,\eta)$ belongs to $\Sigma_t$ of \eqref{structured looping set}. Equivalently, $(y',\omega)$ belongs to $S \Sigma_{-t}$. Hence,
\[
	C_t \subset S \Sigma_{-t}
\]
with the difference being measure $0$.
By our selection of local coordinates, we have that $r^{-1} = p(y',\omega) = 1$, and hence
\[
	|\det \phi''_{(t,r)}| = 1 \qquad \text{ and } \qquad \sgn \phi''_{(t,r)} = 0.
\]
By Lemma \ref{variable very stationary phase lemma}, and recalling that $a$ is a symbol of order $0$, we write \eqref{tr stationary phase prep} as
\begin{align*}
	&(2\pi)^{-\frac14 - \frac{3n}{4} + \frac{d}{2}} \lambda^{n-d-1} \sum_t e^{it\lambda} \iint_{C_t} a(t,y,\lambda \omega) \widehat \chi(t) |g_H(y')|^\frac14 \, d\omega \, dy' + o(\lambda^{n-d-1}) \\
	&= (2\pi)^{-n+d} \lambda^{n-d-1} \sum_t e^{it\lambda} \widehat \chi(t) \iint_{S \Sigma_{-t}} i^{\sigma_{-t}} \sqrt{J_{-t}(y,\eta)} |g_H(y')|^\frac12 \, d\omega \, dy' + o(\lambda^{n-d-1})\\
	&= \lambda^{n-d-1} \sum_t e^{it\lambda} \widehat \chi(t) q(-t) + o(\lambda^{n-d-1}).
\end{align*}
Lemma \ref{A lem} follows after reindexing the sum. 

\begin{remark}\label{maslov remark} We now remark on the nature of the Maslov factor $\sigma_t$ appearing in Lemma \ref{local S_A}. First, consider the distribution $S_A(t_0,y)$ in $y$ for fixed $t_0$, which on one hand is given locally by
\[
	S_A(t_0, y) = (2\pi)^{- \frac{3n}{4} + \frac{d}{2}} \left( \int_{\R^{n-d}} e^{(i\langle y'', \eta'' \rangle - \psi(t_0,y',\eta''))} (2\pi)^{-\frac14} a(t_0,y,\eta'') \, d\eta'' \right) |dy|^\frac12,
\]
and on the other hand is written globally as $e^{-it_0 P} \delta_H$. With a computation similar to that in Lemma \ref{delta_H symbol}, we have a principal symbol 
\[
	i^{\sigma_{-t_0}} (2\pi)^{-\frac n 4 + \frac d 2} \sqrt{J_{-t_0}} \frac{|g_H(y')|^\frac12}{|g(y)|^\frac14} 
\]
wherever $(y',\eta'')$ satisfies \eqref{maximal intersection}. 

Next, suppose that $(x,\xi) \in \Sigma_{t_0}$ and $G^{t_0}(x,\xi) = (y,\eta) \in \Sigma_{s_0}$. Then, $(x,\xi) \in \Sigma_{t_0 + s_0}$ and we write $G^{t_0+s_0}(x,\xi) = (z,\zeta)$. We then compute the principal symbol $a(t,z,\zeta'')$ in two different ways. First, by the computation above, we have
\[
	(2\pi)^{-\frac14} a(t_0 + s_0, z, \zeta'') = i^{\sigma_{-t_0 - s_0}} (2\pi)^{-\frac n 4 + \frac d 2} \sqrt{J_{-t_0 - s_0}} \frac{|g_H(z')|^\frac12}{|g(z)|^\frac14}.
\]
Second, we have
\[
	e^{-(t_0 + s_0)P} \delta_H = e^{-s_0 P} (e^{-t_0 P} \delta_H)
\]
and note that by \eqref{maximal intersection} the tangent space of the Lagrangian associated to $e^{-t_0 P} \delta_H$ agrees with that of $\dot N^*H$, and hence has principal symbol which is $i^{\sigma_{-t_0}} \sqrt{J_{-t_0}}$ times that of $\delta_H$ at $(y,\eta)$. By the calculus in e.g. \cite{HormanderIV} and another application of the calculation above, we also have the local principal symbol expression
\begin{align*}
	(2\pi)^{-\frac14} a(t_0 + s_0, z, \zeta'') &= i^{-\sigma_{-s_0}(z,\zeta)} a(t_0, y, \xi'') \\
	&= (2\pi)^{-\frac n 4 + \frac d 2} i^{-\sigma_{-s_0}(z,\zeta)} i^{\sigma_{-t_0}(y,\eta)} \sqrt{J_{-t_0}(y,\eta)} \frac{|g_H(y')|^\frac12}{|g(y)|^\frac14}\\
	&= (2\pi)^{-\frac n 4 + \frac d 2} i^{-\sigma_{-s_0}(z,\zeta)} i^{\sigma_{-t_0}(G^{-t_0}(z,\zeta))} \\
	&\qquad \times \sqrt{J_{-t_0}(G^{-t_0}(z,\zeta))} \sqrt{J_{-s_0}(z,\zeta)} \frac{|g_H(z')|^\frac12}{|g(z)|^\frac14}.
\end{align*}
This verifies Lemma \ref{group-ish} in the proof of Proposition \ref{Q nice}. Note that the proof of Proposition \ref{off zero singularities} only relies on the definition of $Q$ as a distribution and not that $Q$ is nearly monotonic or bounded.
\end{remark}


\section{Proof of Proposition \ref{Q nice}}
\label{Q NICE}

We first establish the following group-like structure. In what follows, we take $S\Sigma_0 = SN^*H$.

\begin{lemma}\label{group-ish} Let $s$ and $t$ be real numbers and $(x,\xi) \in S\Sigma_t$ and $G^t(x,\xi) \in S\Sigma_s$. Then, at $(x,\xi)$, we have
	\[
	i^{\sigma_{t+s}(x,\xi)} \sqrt{J_{t + s}(x,\xi)} = i^{\sigma_s(G^t(x,\xi))} \sqrt{J_s(G^t(x,\xi))} \cdot i^{\sigma_t(x,\xi)} \sqrt{J_s(x,\xi)}.
	\]
\end{lemma}

\begin{proof}
	That $J_{t + s}(x,\xi) = J_s(G^t(x,\xi)) \cdot J_s(x,\xi)$ is clear. that
	\[
	\sigma_s(G^t(x,\xi)) + \sigma_t(x,\xi) = \sigma_{s + t}(x,\xi) \mod 4
	\]
	follows from Remark \ref{maslov remark} after the computation of $Q$ from the Fourier transform of $N'$.
\end{proof}

Let $S\Sigma_+ = \bigcup_{t > 0} S\Sigma_t$ and let $\mathbf T$ be the first-return time function on $S\Sigma_+$, namely
\[
\mathbf T(x,\xi) = \min \{ t > 0 : (x,\xi) \in S\Sigma_t \}.
\]
We then take an analytic family of operators $U_z : L^2(SN^*H) \to L^2(SN^*H)$ for $\Im z \leq 0$ given by
\[
U_z(f)(x,\xi) =
\begin{cases} e^{-iTz} i^{\sigma_\mathbf T} f \circ G^\mathbf T \sqrt{J_\mathbf T}, &\text{if }(x,\xi)\in SN^*H\cap S\Sigma_+,\\
	0,  &\text{if }(x,\xi)\in SN^*H\setminus S\Sigma_+.
\end{cases}
\]
If $\Im z = 0$, then the operator has norm $\|U_z\| \leq 1$. 
Otherwise if $\Im z < 0$, we have an operator norm bound $\|U_z\| < 1$ since the first-return time is uniformly bounded away from $0$.

\begin{lemma}\label{pos definite} For $\Im z < 0$, the operator
	\[
	I + \sum_{k = 1}^\infty (U_z^k + (U_z^k)^*)
	\]
	is positive-definite.
\end{lemma}

The proof of this lemma is identical to the argument in the proof of \cite[Lemma 1.8.11]{Saf}, but we state it here anyways.

\begin{proof}
	We write
	\begin{align*}
		I + \sum_{k = 1}^\infty (U_z^k + (U_z^k)^*) &= -I + (I - U_z)^{-1} + (I - U_z^*)^{-1} \\
		&= (I - U_z)^{-1}(U_z + (I - U_z)(I - U_z^*)^{-1}) \\
		&= (I - U_z)^{-1}(I - U_z U_z^*)(I - U_z^*)^{-1},
	\end{align*}
	which is a conjugation of a positive-definite operator $I - U_z U_z^*$ by an invertible operator, and hence is also positive-definite.
\end{proof}

Lemma \ref{group-ish} with $s = -t$ implies $q(-t) = \overline{q(t)}$, and hence we write
\begin{equation}\label{Q' pos measure}
	Q'(\lambda) = \sum_{t\in\mathcal T} e^{-it\lambda} q(t) = 2 \Re \sum_{t\in\mathcal T\cap(0,\infty)} e^{-it \lambda} q(t) = \lim_{\epsilon \searrow 0} 2 \Re \sum_{t\in\mathcal T\cap(0,\infty)} e^{-it(\lambda - i\epsilon)} q(t)
\end{equation}
interpreted as a distribution in $\lambda$. By the same lemma and Fubini's theorem,
\[
\sum_{t\in\mathcal T\cap(0,\infty)} e^{-itz} q(t) = (2\pi)^{-n+d} \sum_{k = 1}^\infty \langle U^k_z \mathbf 1, \mathbf 1\rangle,
\]
where here $\mathbf 1$ denotes the constant $1$ function on $S\Sigma_+$. Hence for $\Im z < 0$,
\[
2 \Re \sum_{t\in\mathcal T\cap(0,\infty)} e^{-itz} q(t) = (2\pi)^{-n+d} \sum_{k = 1}^\infty \langle(U^k_z + U_z^{*k}) \mathbf 1, \mathbf 1\rangle.
\]
By Lemma \ref{pos definite}, we have
\[
(2\pi)^{-n+d} |SN^*H| + 2 \Re \sum_{t\in\mathcal T\cap(0,\infty)} e^{-itz} q(t) > 0 \qquad \Im z < 0.
\]
Recalling \eqref{Q' pos measure}, we find that in the limit $\epsilon \searrow 0$, we have that
\[
(2\pi)^{-n+d} |SN^*H| + Q'(\lambda) = (2\pi)^{-n+d} |SN^*H| + \sum_{t\in\mathcal T} e^{-it\lambda} q(t) \geq 0,
\]
and hence is a positive Radon measure. It follows that the function
\[
F(\lambda):=	(2\pi)^{-n+d} |SN^*H| \lambda + Q(\lambda) = (n-d)C_{H,M} \lambda + Q(\lambda)
\]
is a monotone increasing function in $\lambda$, as desired. 

To see that $Q$ is bounded, we select a nonnegative, even, Schwartz-class function $\rho$ with $\int \rho = 1$ with small Fourier support so that $\rho * Q = 0$. Then,
\[
F * \rho(\lambda) = (n-d)C_{H,M} \lambda \qquad \text{ and } \qquad F' * \rho(\lambda) = (n-d)C_{H,M}.
\]
Since $F$ is monotone increasing and tempered, we have by a basic Tauberian theorem such as \cite[Theorem B.2.1]{Saf} that
\[
F(\lambda) = (n-d)C_{H,M} \lambda + O(1),
\]
from which it follows that $Q$ is bounded.


\section{Proof of Results in Section \ref{implications}}

\subsection{Proof of Proposition \ref{recurrent}} 

We first put a metric $d$ on $SN^*H$ which is comparable to the Euclidean metric in each coordinate chart of $SN^*H$. Then, we define
\[
	\mathcal R^\delta = \{(x,\xi) \in SN^*H : d((x,\xi), G^t(x,\xi)) < \delta \text{ for some $t \neq 0$} \}.
\]
We note that $\mathcal R^\delta$ is a monotone-decreasing family of sets as $\delta$ decreases to $0$, and hence for each $\epsilon > 0$, there exists $\delta$ for which
\[
	|\mathcal R^\delta| < \frac\epsilon 2,
\]
where the measure is taken with respect to the natural measure on $SN^*H$ discussed in the introduction. We then take an open neighborhood $U$ of $\mathcal R^\delta$ with $|U| < \epsilon$. We claim there exists an integer $N_\delta$ depending on $\delta$ such that, for each $(x,\xi)$ in the complement $U^c$ of $U$,
\[
	\#\{ t \neq 0 : G^t(x,\xi) \in U^c \} \leq N_\delta.
\]
To see this, cover $U^c$ by a minimal collection of balls $B_1,\ldots,B_{N_\delta}$ of radius $\delta/2$. Note, if $(x,\xi) \in B_j\cap U^c$, then $d((x,\xi), G^t(x,\xi)) \geq \delta$ for all $t \neq 0$, and hence $G^t(x,\xi) \not\in B_j$ for any $t \neq 0$. It follows then that, given any choice of $(x,\xi)$, $G^t(x,\xi)$ belongs to each of $B_1\cap U^c,\ldots,B_{N_\delta}\cap U^c$ at at most one time each. Since $B_1,\ldots,B_{N_\delta}$ covers $U^c$, the claim follows. We will also repeatedly use the following lemma.

\begin{lemma}\label{cauchy-schwarz} For each $t \neq 0$ and nonnegative measurable functions $f$ and $g$ on $S\Sigma_t$,
\[
	\int_{S\Sigma_t} fg \sqrt{J_t} \leq \left( \int_{S\Sigma_t} f^2 \right)^\frac12 \left( \int_{S\Sigma_{-t}} (g \circ G^{-t})^2 \right)^\frac12.
\]
\end{lemma}

\begin{proof} By Cauchy-Schwarz, the left side is bounded by
\[
	\left( \int_{S\Sigma_t} f^2 \right)^\frac12 \left( \int_{S\Sigma_{t}} g^2 J_t \right)^\frac12.
\]
The lemma follows after performing a change of variables $(x,\xi)$ to $(y,\eta)$ via $(y,\eta) = G^t(x,\xi)$ on the rightmost integral.
\end{proof}

Now,
\[
	|q(t)| \leq \int_{S\Sigma_t} \mathbf 1_{U} \sqrt{J_t} + \int_{S\Sigma_t} \mathbf 1_{U^c} (\mathbf 1_{U} \circ G^t) \sqrt{J_t} + \int_{S\Sigma_t} \mathbf 1_{U^c} (\mathbf 1_{U^c} \circ G^t) \sqrt{J_t}.
\]
Applying Lemma \ref{cauchy-schwarz} to the first integral yields
\[
	\int_{S\Sigma_t} \mathbf 1_{U} \sqrt{J_t} \leq |U|^\frac12 |SN^*H|^\frac12 < \epsilon^\frac12 |SN^*H|^\frac12.
\]
Another application to the second integral yields
\[
	\int_{S\Sigma_t} \mathbf 1_{U^c} (\mathbf 1_{U} \circ G^t) \sqrt{J_t} \leq \left( \int_{S\Sigma_t} \mathbf 1_{U^c} \right)^\frac12 \left( \int_{S\Sigma_{-t}} \mathbf 1_{U} \right)^\frac12 \leq |SN^*H|^\frac12 \epsilon^\frac12.
\]
It suffices now to show that
\[
	\lim_{T \to \infty} \frac1T \sum_{t\in\mathcal T\cap[-T,T]} \int_{S\Sigma_t} \mathbf 1_{U^c} (\mathbf 1_{U^c} \circ G^t) \sqrt{J_t} = 0,
\]
since then we will have shown the limit (more rigorously, the limit supremum) in \eqref{averaging condition} is bounded by a constant times $\epsilon^\frac12$, where here $\epsilon > 0$ is arbitrarily small. By the lemma and the inequality $ab \leq \frac12 a^2 + \frac12 b^2$, we have
\begin{align*}
	\int_{S\Sigma_t} \mathbf 1_{U^c} (\mathbf 1_{U^c} \circ G^t) \sqrt{J_t} &\leq \frac12 \int_{S\Sigma_t} \mathbf 1_{U^c} (\mathbf 1_{U^c} \circ G^t) + \frac12 \int_{S\Sigma_{-t}} (\mathbf 1_{U^c} \circ G^{-t}) \mathbf 1_{U^c} \\
	&\leq \frac12 |U^c \cap G^{-t}(U^c)| + \frac12 |U^c \cap G^{t}(U^c)|,
\end{align*}
and hence
\[
	\frac1T \sum_{t\in\mathcal T\cap[-T,T]} \int_{S\Sigma_t} \mathbf 1_{U^c} (\mathbf 1_{U^c} \circ G^t) \sqrt{J_t} \leq \frac{1}{T} \sum_{t\in\mathcal T\cap[-T,T]} |U^c \cap G^{-t}(U^c)|.
\]
By the claim, $G^t(x,\xi) \in U^c$ for at most $N_\delta$ nonzero times $t$. Hence we have
\begin{align*}
	\frac{1}{T} \sum_{t\in\mathcal T\cap[-T,T]} |U^c \cap G^{-t}(U^c)| &\leq \frac{1}{T} \sum_{t\in\mathcal T\cap[-T,T]} \int_{SN^*H} \mathbf 1_{U^c}(G^t(x,\xi)) \, dx \, d\xi \\
	&= \frac1T \int_{SN^*H} \sum_{t\in\mathcal T\cap[-T,T]} \mathbf 1_{U^c}(G^t(x,\xi)) \, dx \, d\xi \leq \frac{|SN^*H| N_\delta}{T},
\end{align*}
which indeed vanishes as $T \to \infty$ as desired.

\subsection{Proof of Theorem \ref{invariant function}}

	We consider the time-$t$ return operator $U_t$ on $L^2(SN^*H)$ given by
	\[
		U_t f = \mathbf 1_{S\Sigma_t} f \circ G^t \sqrt{J_t}.
	\]
	Note, $\langle U_t \mathbf 1, \mathbf 1\rangle$ is similar to the integral in the definition of $q(t)$ except for the missing the Maslov factor. Lemma \ref{pos definite}, or rather its proof, yields that for each $f \in L^2(SN^*H)$,
	\[
		\|f\|_{L^2(SN^*H)}^2 + \sum_{t \in \mathcal T} e^{-it\lambda} \langle U_t f, f\rangle
	\]
	is a nonnegative Radon measure in $\lambda$. In particular, if $\rho$ is a nonnegative Schwartz-class function on $\R$,
	\[
		A_\rho := \widehat \rho(0) I + \sum_{t \in \mathcal T} \widehat \rho(-t) U_t
	\]
	is a positive-definite operator on $L^2(SN^*H)$. Furthermore, one verifies that $U_{-t} = U_t^*$, and hence since $\widehat \rho$ is even, $A_\rho$ is also self-adjoint. It follows that $\langle A_\rho f,f\rangle^\frac12$ is a seminorm on $L^2(SN^*H)$, and hence we have the triangle inequality
	\[
		\langle A_\rho (f + g),f + g\rangle ^\frac12 \leq \langle A_\rho f, f\rangle ^\frac12 + \langle A_\rho g, g\rangle ^\frac12.
	\]
	
	Next, denote 
	\[
	S\Sigma_{+\infty}:=\limsup_{t\rightarrow+\infty} S\Sigma_t=\bigcap_{T=0}^\infty\bigcup_{t>T}S\Sigma_t,
	\]
	which is the subset of $(x,\xi)\in S\Sigma_+$ that loops back normally to $H$ infinitely often via the geodesic flow. Similarly, define
	\[
	S\Sigma_{-\infty}:=\limsup_{t\rightarrow-\infty} S\Sigma_t=\bigcap_{T=0}^\infty\bigcup_{t<-T}S\Sigma_t,
	\]
	and let
	\[
		S\Sigma_\infty = S\Sigma_{+\infty} \cap S\Sigma_{-\infty}
	\]
	denote the set of directions which loop back for infinitely many positive and infinitely many negative times. We partition $SN^*H$ into sets $S\Sigma_\infty$, $B_1$, and $B_2$ with $B_1 = S\Sigma_{-\infty} \setminus S\Sigma_{+\infty}$ and $B_2 = SN^*H \setminus ( S\Sigma_{\infty}\cup B_1)$. That is, $B_1$ are directions that loop back for infinitely many negative times but only for finitely many positive times. $B_2$ are all directions that only loop back for finitely many negative times. Now we assert that $\widehat \rho \geq 0$ and $\widehat \rho > 0$ on $[-1,1]$ and $\supp \widehat \rho \subset [-2,2]$ and take the (unusual) scaling
	\[
		\rho_T(\tau) = \rho(T\tau).
	\]
	By the triangle inequality, we have
	\begin{multline*}
		\frac1T \left( |SN^*H| + \sum_{t \in \mathcal T} \widehat \rho(-t/T) |q(t)| \right) \leq \langle A_{\rho_T}(\mathbf 1\rangle , \mathbf 1\rangle ^\frac12 \\
		\leq \langle A_{\rho_T}(\mathbf 1_{S\Sigma_\infty}) , \mathbf 1_{S\Sigma_\infty}\rangle ^\frac12 + \langle A_{\rho_T}(\mathbf 1_{B_1}), \mathbf 1_{B_1}\rangle ^\frac12 + \langle A_{\rho_T}(\mathbf 1_{B_2}), \mathbf 1_{B_2}\rangle ^\frac12.
	\end{multline*}
	We claim that
	\begin{equation}\label{invariant claim 1}
		\lim_{T \to \infty} \langle A_{\rho_T}(\mathbf 1_{B_i}), \mathbf 1_{B_i}\rangle  = 0 \qquad i = 1,2
	\end{equation}
	essentially by construction. We then also claim that
	\begin{equation}\label{invariant claim 2}
		\lim_{T \to \infty} \langle A_{\rho_T}(\mathbf 1_{S\Sigma_\infty}), \mathbf 1_{S\Sigma_\infty}\rangle  = 0
	\end{equation}
	under the assumptions in the theorem. \eqref{averaging condition} follows.
	
	We now proceed with the proof of \eqref{invariant claim 1}. Now, since $U_{-t} = U^*_t$ and $\widehat \rho$ is even, we have
	\begin{align*}
		\langle A_{\rho_t} \mathbf 1_{B_i}, \mathbf 1_{B_i}\rangle  &= \frac{1}{T} \left(|B_i| + \sum_{t \in \mathcal T} \widehat \rho(-t/T) \langle U_t \mathbf 1_{B_i}, \mathbf 1_{B_i}\rangle  \right) \\
		&= \frac{1}{T} \left(|B_i| + \sum_{t > 0} \widehat \rho(-t/T) \langle U_t \mathbf 1_{B_i}, \mathbf 1_{B_i}\rangle  + \sum_{t < 0} \widehat \rho(-t/T) \langle\mathbf 1_{B_i}, U_{-t} \mathbf 1_{B_i}\rangle   \right) \\
		&= \frac{1}{T} \left(|B_i| + 2 \sum_{t > 0} \widehat \rho(-t/T) \langle U_t \mathbf 1_{B_i}, \mathbf 1_{B_i}\rangle  \right) \\
		&\leq \frac{1}{T} \left(|B_i| + 2 \sum_{0 < t < 2T} \langle U_t \mathbf 1_{B_i}, \mathbf 1_{B_i}\rangle  \right).
	\end{align*}
	By Lemma \ref{cauchy-schwarz}, we have
	\[
		\langle U_t \mathbf 1_{B_i}, \mathbf 1_{B_i}\rangle  = \int_{S\Sigma_t} \mathbf 1_{B_i} (\mathbf 1_{B_i} \circ G^t) \sqrt{J_t} \leq |B_i \cap S\Sigma_t|^\frac12 |B_i \cap S\Sigma_{-t}|^\frac12.
	\]
	By Cauchy-Schwarz, we have
	\[
		\frac1T \sum_{0 < t < 2T} \langle U_t \mathbf 1_{B_i}, \mathbf 1_{B_i}\rangle  \leq \left( \frac1T \sum_{0 < t < 2T} |B_i \cap S\Sigma_t| \right)^\frac12 \left( \frac1T \sum_{0 < t < 2T} |B_i \cap S\Sigma_{-t}| \right)^\frac12.
	\]
	In the case $i = 1$, the latter quantity in parentheses is bounded in $T$, and the former can be written as
	\[
		\int_{B_1} \frac{\#\{ t \in \mathcal T \cap (0,2T] : (x,\xi) \in S\Sigma_t \}}{T} \, dx \, d\xi,
	\]
	which vanishes in the limit $T \to \infty$ by Lebesgue dominated convergence since the integrand is bounded in $T$ and vanishes in the limit by construction. The case $i = 2$ is similar, and we have \eqref{invariant claim 1}
	
	We are ready to prove \eqref{invariant claim 2}. Now for each $(x,\xi) \in SN^*H$, let $\mathbf T(x,\xi)$ be the first positive time $t$ for which $(x,\xi) \in S\Sigma_t$, with $\mathbf T(x,\xi) = \infty$ if no such time exists. Define an operator $U : L^2(SN^*H) \to L^2(SN^*H)$ by
	\[
	U(f)(x,\xi) =
	\begin{cases} f \circ G^\mathbf T \sqrt{J_\mathbf T}, &\text{if }(x,\xi)\in SN^*H\cap S\Sigma_+,\\
		0,  &\text{if }(x,\xi)\in SN^*H\setminus S\Sigma_+.
	\end{cases}
	\]
	Assume after perhaps rescaling the metric on $M$ that $\mathbf T(x,\xi) \geq 1$ for each $(x,\xi) \in SN^*H$. Then,
	similar to the reduction from before,
	\begin{align*}
		\langle A_{\rho_T} \mathbf 1_{S\Sigma_\infty}, \mathbf 1_{S\Sigma_\infty}\rangle  &\leq \frac1T \left( |S\Sigma_\infty| + 2\sum_{0 < t < 2T} \langle U_t \mathbf 1_{S\Sigma_\infty}, \mathbf 1_{S\Sigma_\infty}\rangle  \right) \\
		&\leq \frac2T \left( \sum_{k = 0}^{\lfloor 2T \rfloor} \langle U^k \mathbf 1_{S\Sigma_\infty}, \mathbf 1_{S\Sigma_\infty}\rangle  \right).
	\end{align*}

	To show this vanishes in the limit $T \to \infty$, we use von Neumann's mean ergodic theorem. To do so, we first verify:
	\begin{lemma}
			The restriction of $U$ to $L^2(S\Sigma_\infty)$ is a Hilbert space isometry.
		\end{lemma}
		
		\begin{proof}
			List the distinct first-return times of points on $S\Sigma_\infty$ as $T_1,T_2,\ldots$ and write $Uf$ as the countable sum $\sum_j U_jf$ where $U_jf(x,\xi)= f \circ G^T \sqrt{J_T} \mathbf 1_{\mathbf T = T_j}$. We claim that
			\[
			\|Uf\|^2_{L^2(S\Sigma_\infty)}=\sum_j\|U_j f\|^2_{L^2(S\Sigma_\infty)} = \|f\|^2_{L^2(S\Sigma_\infty)}
			\]
			for each $j = 1,2,\ldots$. The first equality follows from the fact that $Uf_1,Uf_2,\ldots$ are mutually orthogonal by definition. To see the second equality, we write $\sum_j\|U_j f\|^2_{L^2(S\Sigma_\infty)}$ as
			\[
			\sum_j\int_{S\Sigma_{T_j}} |f \circ G^{T_j}|^2  \mathbf 1_{\mathbf T = T_j} J_{T_j} = \sum_j\int_{S\Sigma_{-T_j}} |f|^2  \mathbf 1_{\mathbf T = T_j}\circ G^{-T_j},
			\]
			where we have used the standard change of variables. But, $(x,\xi) \in S\Sigma_{-T_j}$ if and only if $(x,-\xi) \in S\Sigma_{T_j}$, and furthermore $(x,\xi) \mapsto (x,-\xi)$ is a measure-preserving map $S\Sigma_{-T_j} \to S\Sigma_{T_j}$. Finally, since $\mathbf T(x,\xi)=T_j$ if and only if $\mathbf T(G^{-T_j}(x,-\xi))=T_j$, the right hand side equals
			\begin{multline}
				\sum_j\int_{S\Sigma_{-T_j}} |f|^2(x,\xi)  \mathbf 1_{\mathbf T = T_j} (x,-\xi)=\sum_j\int_{S\Sigma_{T_j}} |f|^2(x,-\xi)  \mathbf 1_{\mathbf T = T_j} (x,\xi)\\=\int_{S\Sigma_{\infty}} |f|^2(x,-\xi)=\int_{S\Sigma_{\infty}} |f|^2(x,\xi).
			\end{multline}
			The last equality is due to the fact that $(x,\xi) \in S\Sigma_{\infty}$ if and only if $(x,-\xi) \in S\Sigma_{\infty}$, and furthermore $(x,\xi) \mapsto (x,-\xi)$ is a measure-preserving map on $S\Sigma_{\infty}$. Our claim is proved.
	\end{proof}
	
	Now, by von Neumann's mean ergodic theorem, 	
	\[
		\lim_{T \to \infty} \frac2T \left( \sum_{k = 0}^{\lfloor 2T \rfloor} \langle U^k \mathbf 1_{S\Sigma_\infty}, \mathbf 1_{S\Sigma_\infty}\rangle  \right) = 4 \mathit\Pi(\mathbf 1) 
	\]
	where $\mathit\Pi(\mathbf 1)$ denotes the projection of $\mathbf 1$ onto the $U$ invariant subspace of $L^2(S\Sigma_\infty)$. By our assumption that the geodesic flow has no $L^1$ invariant measure on $S\Sigma_\infty$,  $\mathit\Pi(\mathbf 1)=0$. The proof is complete.

\section*{Declarations}

\subsection*{Funding}

Xi was supported by the National Key R\&D Program of China, No: 2022YFA1007200, NSF China Grant No. 12171424, and the Fundamental Research Funds for the Central Universities 2021QNA3001.

Wyman was partially supported by NSF Grant DMS-2204397 the AMS-Simons Travel Grants.

\subsection*{Financial and non-financial interests}

The authors have no financial or proprietary interests in any material discussed in this article.

\bibliography{references}{}
\bibliographystyle{alpha}

\end{document}